\documentclass[a4paper,10pt]{article}
\usepackage[top=2.54cm,bottom=2.0cm,left=2.0cm,right=2.54cm, includeheadfoot]{geometry}

\usepackage[T1]{fontenc}
\usepackage[utf8]{inputenc}
\usepackage[english]{babel}
\usepackage{enumerate}
\usepackage{multirow,booktabs}
\usepackage[table]{xcolor}
\usepackage{fullpage}
\usepackage{lastpage}
\usepackage{indentfirst}

\usepackage{footnote}

\usepackage{amsmath,amsfonts,amssymb,amscd,amsthm}

\usepackage{bm}

\usepackage[all,2cell]{xy} \UseAllTwocells \SilentMatrices

\usepackage{mathpartir}

\usepackage{hyperref}

\usepackage{subfiles}

\usepackage{multicol}

\usepackage{lineno}

\newtheorem{theorem}{Theorem}[section]
\newtheorem{lemma}[theorem]{Lemma} 
\newtheorem{corollary}[theorem]{Corollary}
\newtheorem{proposition}[theorem]{Proposition} 

\newtheorem{definition}[theorem]{Definition}

\newtheorem*{claim*}{Claim}

\newtheorem{remark}[theorem]{Remark}

\def\La{\ensuremath{\mathsf L}}

\def\Sw{\mathsf S}

\newcommand{\iplp}{\ensuremath{IPL^+}}

\begin{document}

\title{Hyper swap structures and Kalman functors: the case study of da Costa logic $C_\omega$}

\author{
 Marcelo E. Coniglio $^\textup{\scriptsize a,b}$
   \and
   Kaique Roberto $^\textup{\scriptsize a,c}$ \and Ana Claudia Golzio
 $^\textup{\scriptsize d}$
}

\date{
   $^\textup{\scriptsize a}$\textit{\small Centro de L\'ogica, Epistemologia e Hist\'oria da Ci\^encia (CLE), Universidade Estadual de Campinas (UNICAMP), Campinas, Brazil}
   \\
   $^\textup{\scriptsize b}$\textit{\small Instituto de Filosofia e Ci\^{e}ncias Humanas (IFCH),  Universidade Estadual de Campinas (UNICAMP), Campinas, Brazil}
   \\
   $^\textup{\scriptsize c}${\small   Faculdade Israelita de Ciências da Saúde Albert Einstein (FICSAE), S\~ao Paulo, Brazil}
    \\
    $^\textup{\scriptsize d}$\textit{\small Faculdade de Ciências e Engenharia (FCE),  Universidade Estadual Paulista (UNESP), Tupã, Brazil}
}

\maketitle

\begin{abstract}
In a previous paper, we recast Morgado hyperlattices and Sette implicative hyperlattices in lattice-theoretic terms. By utilizing swap structures induced by implicative lattices, we obtained a direct proof of soundness and completeness for da Costa’s paraconsistent logic $C_\omega$ with respect to Sette's hyperalgebraic semantics. Inspired by Kalman functors in the context of twist structures, we introduce the notion of hyper swap structures, a novel class of hyperalgebras that naturally generalize swap structure semantics. We prove that these hyperalgebras, besides providing another class of hyperalgebraic models for  $C_\omega$, induce a Kalman-style functor between the category of Sette implicative hyperlattices and the category of enriched hyperalgebras for $C_\omega$. Specifically, we exhibit an equivalence of categories between Sette implicative hyperlattices and their enriched hyperalgebraic counterparts using Kalman and forgetful functors. Similar results are extended to two axiomatic extensions of $C_\omega$.
\end{abstract}

\section{Introduction}\label{sec:introd}

In a seminal paper from 1958,  J.~A. Kalman~\cite{kalman:1958} introduced a simple but highly original algebraic construction. He observed that, given a bounded distributive lattice $\mathcal{L}=\langle L,\land,\lor,0,1\rangle$, the set $K(\mathcal{L})=\{(a,b) \in L^2 \ : \ a \land b =0\}$ forms a centered Kleene algebra with the following operations:
$$(a,b) \,\bar{\land}\, (c,d) = (a \land c, b \vee d)$$
$$(a,b) \,\bar{\lor}\, (c,d) = (a \lor c, b \land d)$$
$$\bar{\neg} (a,b) = (b, a).$$

\noindent The center of $K(\mathcal{L})$ (i.e., the unique element $c$ such that $\bar{\neg} c=c$) is $(0,0)$. In 1986, R. Cignoli~\cite{cign:86} observed that the mapping $K$ can be extended to morphisms  by defining $K(f):K(\mathcal{L}) \to K(\mathcal{L}')$ as $K(f)(a,b)=(f(a),f(b))$, for every lattice homomorphism $f:\mathcal{L} \to \mathcal{L}'$. This gives rise to a functor $K$ (the {\em Kalman functor}), from the category of bounded lattices to the category of centered Kleene algebras, which has a left adjoint. Furthermore, by considering the full subcategory of Kleene centered algebras satisfying an interpolation property (which M. Sagastume proved to be equivalent to an algebraic condition, see~\cite{jansana:sanmart:2018} for details), Cignoli  obtained an equivalence of categories.  In addition, he  observed that, in an independent way,  M. Fidel~\cite{fidel:1978} and D. Vakarelov~\cite{vakarelov:1977} also introduced the Kalman's construction $K(\mathcal{H})$ for any Heyting algebra $\mathcal{H}$, obtaining in this particular case Nelson algebras. Cignoli also showed that the Kalman functor establishes an equivalence between the category of Heyting algebras and the category of centered Nelson algebras, which allows us to study Nelson algebras in terms of twist structures over Heyting algebras. After M.  Kratch~\cite{kracht:1998}, this kind of Kalman constructions is nowadays known as {\em twist structures}, and the matrix semantics associated with a class of twist structures for a given logic is called {\em twist structures semantics}. In such matrices, the set of designated elements is usually defined by $D = \{(a,b) \ : \ a=1\}$. Kalman construction, as well as its associated Kalman functor and the induced twist structures semantics, were afterwards adapted to several algebraic contexts and different logic systems, see for instance~\cite{odin:08, Rivieccio2014, casti:cel:sanmart:2017, jansana:sanmart:2018} and, more recently, \cite{busa:gala:mar:2022} (and the references therein).

Despite the theoretical and practical interest of obtaining an algebraic counterpart to logic systems (the subject studied in the so-called  {\em Abstract Algebraic Logic}, AAL --- for an excellent textbook, see \cite{font:16}), it is well known that several classes of logic systems cannot be characterized in algebraic terms, at least by means of the traditional tools of AAL. For instance, many systems in the class of paraconsistent logics known as {\em Logics of Formal Inconsistency} (in short LFIs, see for instance~\cite{carnielli2016paraconsistent}) lie outside the scope of the usual techniques of AAL. In addition, some of such systems cannot be semantically characterized by a single finite logical matrix. This forces us to search for new kind of semantics, in general non-deterministic:  (non-truth-functional) bivaluations, possible-translations semantics, Fidel structures, and non-deterministic matrices (or Nmatrices), obtaining in several cases decision procedures for these logics. 

Nmatrices generalize logical matrices by replacing the underlying algebra of truth-values by a hyperalgebra, i.e., a structure in which the  connectives are interpreted as hyperoperators that associate to each input a non-empty set of possible values. They were formally introduced in 2001 by A. Avron and I. Lev (see~\cite{Av:01}), although Nmatrices were already considered in the literature several decades before that. For instance, Yu. Ivlev studied several systems of  non-normal modal logics with finite-valued Nmatrix semantics   (a recent survey of Ivlev's contributions to modal logics with Nmatrix semantics can be found in~\cite{Ivlev:2024}).

In~\cite[Chapter~6]{carnielli2016paraconsistent} a systematic way was introduced to define Nmatrices in an analytical way, through the notion of {\em swap structures}.   These hyperalgebras can be seen as non-deterministic twist structures, since  the elements of their domain are $(n+1)$-tuples over a given ordered algebra (in general a Boolean algebra) $\mathcal{A}$ in which its coordinates represent a truth-value (in  $\mathcal{A}$) assigned to $\varphi$, $\alpha_1(\varphi)$, \ldots, $\alpha_{n}(\varphi)$. Each $\alpha_i(p)$ is a formula representing, in general,  a (non truth-functional) connective of the logic being characterized, which does not have an algebraic interpretation in $\mathcal{A}$. In~\cite{coniglio2019swap}, swap structures for Ivlev-like modal logics were proposed for the first time. In~\cite{coniglio2020non}, swap structures were introduced for several LFIs, and a generalization of the Kalman functor was also considered: for each LFI, the functor produces, for a given Boolean algebra $\mathcal{A}$, a swap structure $K(\mathcal{A})$. Although the functor $K$ preserves products and monomorphisms, the existence of a left adjoint for $K$ does not appear to be possible to obtain.

In a recent preprint~\cite{con:gol:rob:2025}, we characterized da Costa paraconsistent logic $C_\omega$ in terms of swap structures defined over implicative lattices, which are the algebraic counterpart of positive intuitionistic logic \iplp. This is justified by the fact that $C_\omega$ extends \iplp\ by adding a paraconsistent negation, and it is not finitely trivializable (that is, $C_\omega$ cannot define a bottom formula $\bot$, and thus it is not an LFI). Our characterization was done in terms of hyperalgebras for  $C_\omega$ expanding  Morgado hyperlattices (introduced in~\cite{morgado1962introduccao}) and Sette implicative hyperlattices (introduced in~\cite{sette1971algebras}), which constitute a very natural transposition of the notions of lattices and implicative lattices to the realm of hyperstructures. The class of swap structures for  $C_\omega$ induced by implicative lattices is contained in  the class  $\mathbb{HC}_\omega$ of hyperalgebras for  $C_\omega$, such that this logic is sound and complete w.r.t. swap structures semantics, as well as w.r.t. Nmatrix semantics over $\mathbb{HC}_\omega$. In this way, we obtain a Kalman-style functor from the category of implicative lattices into the category   $\mathbb{HC}_\omega$ of hyperalgebras for  $C_\omega$. While we have successfully abstracted swap structures to a class of hyperalgebras, finding a left adjoint to this functor  does not seem possible, in principle --- a situation analogous to the LFIs  case treated  in~\cite{coniglio2020non}.

A possible explanation for the lack of existence of an adjoint for the Kalman functor for hyperalgebras (through swap structures), compared to the algebraic case (through twist structures) is that, in the latter construction, the Kalman functor relates categories of {\em algebras}. Hence, the adjoint functor `forgets' the additional operator of the more complex algebras (a negation operator, in most cases) and keeps the underlying algebraic structure. In the swap structure construction, in turn, the Kalman functor associates to an {\em algebra} (an implicative lattice in the case of $C_\omega$, or  a Boolean algebra in the case of LFIs) an {\em hyperalgebra} (a swap structure). Hence, any functor in the opposite direction to the Kalman functor should `forget' the new operator(s)  induced in the swap structure (or, in general, in the class of hyperalgebras to which the swap structures belong, if such abstraction is possible) thereby obtaining an {\em algebra}. But this seems to be inadequate, provided that the reduct of the swap structures (and of the hyperalgebras in the full class of models) without such new operator(s) is also a hyperalgebra, not an algebra: information may be lost along this `forgetful' process. Thus, it seems much more reasonable to define a swap structure  from a hyperalgebra with the same logical behavior as the class of algebras originally considered. That is, the Kalman-style functor associated to swap structures should relate categories of hyperalgebras, just as the Kalman functor associated to twist structures relates categories of algebras. For instance, in the case of $C_\omega$, a swap structure should be defined starting from a  Sette implicative hyperlattice, and in the case of LFIs, the swap structures should be defined over Boolean  hyperalgebras, obtaining in both cases Kalman-style functors relating categories of hyperalgebras.

Thus, the aim of this paper is to generalize, to the hyperalgebraic setting, the familiar twist structure technique discussed above, by proposing the novel notion of \emph{hyper swap structures}.  Our main objective is to establish equivalences of categories between a `base’ category of hyperalgebras (which interprets the `standard' operators such as conjunction, disjunction and deductive implication) and an `induced’ category of hyperalgebras (containing `non-standard' non-truth-functional operators such as paraconsistent negations that do not preserve logical equivalences in general) whose \emph{representative} objects are swap structures, thus mirroring what is achieved for algebras via twist constructions.  To fix ideas, and as a case example, we start by replacing the swap structures for $C_\omega$ defined over standard implicative lattices, considered in our previous paper~\cite{con:gol:rob:2025}, with hyper swap structures defined over Sette implicative hyperlattices. Thus, by passing to the subcategory of \emph{enriched hyper $C_{\omega}$ algebras} (Definition~\ref{def:EHCw}), we prove in Section~\ref{sect:equiv01} that the category of Sette implicative hyperlattices is equivalent to that of enriched hyper $C_{\omega}$ algebras, with hyper swap structures serving as the representative objects.  This level of generality --- remaining entirely within hyperalgebra categories --- is essential to lift twist‐based equivalences into the hyperalgebraic realm.

The organization of this paper is as follows: in Section~\ref{sec:morgado} we recall the basic notions and results on hyperlattices that will be employed throughout the paper, in particular the definitions and key properties of Sette implicative hyperlattices presented in~\cite{con:gol:rob:2025}. In Section~\ref{sec:Cw} we recall the paraconsistent logic $C_{\omega}$, which is a kind of `syntactical limit' of da Costa’s hierarchy of systems $C_{n}$ (for $1 \le n < \omega$). Our treatment of $C_{\omega}$ follows the presentations in~\cite{da1993sistemas,carnielli2016paraconsistent}. In Section~\ref{sec:swap-str} it is introduced the notion of hyper swap structures for $C_{\omega}$: we give their formal definition, develop the corresponding semantics, and prove both soundness and completeness of $C_{\omega}$ with respect to this semantics. In Section~\ref{sect:equiv01} we define the class of Enriched Hyper $C_{\omega}$ Algebras (EHC$_{\omega}$A) and establish the central equivalence of categories between the category of Sette implicative hyperlattices (IHLs) and the category of EHC$_{\omega}$As. In Section~\ref{aect:axiom-ext} the results on $C_\omega$ presented in the previous sections are adapted to the logics $C_{min}$ (introduced in~\cite{car:mar:99}) and $C_\omega^+$, two interesting axiomatic extensions of  $C_\omega$.
Finally, in Section~\ref{sect:final-remarks} we summarize our results and outline possibilities for future research.

\section{Morgado hyperlattices and Sette hyperalgebras}\label{sec:morgado}


In this section we recall the basic notions and results on hyperlattices used throughout this paper, which were taken from our previous paper~\cite{coniglio2020non}. 

\begin{definition} \label{def:signature}
 A {\em propositional signature} is a sequence of pairwise (possibly empty) disjoint sets $\Sigma=(\Sigma_n)_{n\in \mathbb{N}}$. Elements in $\Sigma_n$ are called {\em $n$-ary constructors}, and elements in $\Sigma_0$ are called {\em  symbols for constants}. The  free algebra over $\Sigma$ generated by a denumerable set $\mathcal{P}=\{p_0,p_1,\ldots\}$ of  propositional variables will be denoted by $For(\Sigma)$. When $|\Sigma|:= \bigcup_{n \geq 0} \Sigma_n$ is finite then  $\Sigma$ will be  represented by $|\Sigma|$.
\end{definition}

Propositional signatures will be used for describing propositional languages, algebras and hyperalgebras.

\begin{definition} Let $A$ be any set. A {\em hyperoperation of arity $n \in \mathbb{N}$} over $A$ is a function $\#:A^n \to \mathcal \wp (A)\setminus\{\emptyset\}$. If  $\#(\vec a)$ is  a singleton for every $\vec a \in A^n$, then the hyperoperation $\#$ can be naturally identified with an ordinary $n$-ary operation  $\#:A^n \to A$. A $0$-ary hyperoperation on $A$ can be identified with a non-empty subset of $A$.
\end{definition}

\begin{definition}
 A {\em hyperalgebra} over a signature $\Sigma$ is a set $A$ endowed with a family of $n$-ary hyperoperations $\tilde \# : A^n \to \mathcal \wp(A)\setminus\{\emptyset\}$, for every $\# \in \Sigma_n$ and $n \in \mathbb N$.
\end{definition}

\begin{definition}[Prosets] A pre-ordered set (Proset) is a pair $\mathsf P=\langle P, \preceq\rangle$ such that $P$ is a non-empty set and $\preceq$ is a reflexive and transitive relation on $P$. That is, for every $x, y, z \in P$: $x \preceq x$ and $x \preceq y, y \preceq z$ imply $x \preceq z$.

We say that $x$ and $y$ are {\em similar},  denoted by $x \equiv y$, if $x \preceq y$ and $y \preceq x$. For $B,C \subseteq P$, the expression $B \preceq C$ means that $x \preceq y$ for every $x \in B$ and every $y \in C$. Accordingly, $x \preceq B$ denotes that $x \preceq y$ for every $y \in B$, and $B \preceq x$ denotes that $y \preceq x$ for every $y \in B$.
\end{definition}

\noindent
It is worth noting  that $\emptyset \preceq B$ and $B \preceq \emptyset$ for every $B \subseteq P$. Analogously, $x \preceq\emptyset$ and $\emptyset \preceq x$ for every $x \in P$.

\begin{definition}
Let $\mathsf P$ be a proset, and let $B \subseteq P$.
\begin{enumerate}
    \item The set of {\em minima} of $B$ is $\mathsf{Min}(B)=\{x \in B \ : \ x \preceq B\}$, and the  set of {\em maxima} of $B$ is $\mathsf{Max}(B)=\{x \in B \ : \ B \preceq x\}$.
    \item The set of {\em upper bounds} of $B$ is $\mathsf{Ub}(B)=\{z \in P \ : B \preceq z\}$. The set of lower bounds of $B$ is $\mathsf{Lb}(B)=\{z \in P \ : z \preceq B\}$.
\end{enumerate}
\end{definition}

\noindent
As a consequence of the definitions, $\mathsf{Min}(\emptyset)=\mathsf{Max}(\emptyset)=\emptyset$, and  $\mathsf{Ub}(\emptyset)=\mathsf{Lb}(\emptyset)=P$.

In 1962, J. Morgado introduced an interesting notion of hyperlattices based on prosets:

\begin{definition} [Morgado hyperlattices, {\cite[Ch.~II, \S 2, p.~122]{morgado1962introduccao}}] \label{def:m-hlattices}  Let $\mathsf P$ be a proset, and let $x,y \in P$.
\begin{enumerate}
    \item The {\em Morgado hypersupremum} (or {\em supremoid}) of $x$ and $y$ is the set $x \curlyvee y = \mathsf{Min}(\mathsf{Ub}(\{x,y\}))$.
    \item The {\em Morgado hyperinfimum} (or {\em infimoid}) of $x$ and $y$ is the set $x \curlywedge y = \mathsf{Max}(\mathsf{Lb}(\{x,y\}))$.
    
    \item $\mathsf P$ is said to be a {\em Morgado hyperlattice} (or an {\em m-hyperlattice}, or simply an {\em hyperlattice}) if $x \curlyvee y$ and $x \curlywedge y$ are nonempty sets for every $x,y \in P$.
\end{enumerate}
\end{definition}

\begin{definition} [Stable sets, {\cite[Definition~8]{coniglio2020non}}] \label{stable} Let  $\emptyset\neq A, B \subseteq L$. We say that $A$ and $B$ are {\em similar}, and write $A\equiv B$, if $a \equiv b$ for every $a \in A$ and $b \in B$. That is: $a \preceq b$ and $b \preceq a$  for every $a \in A$ and $b \in B$. A non-empty subset $A\subseteq L$ is {\em stable} if $A\equiv A$, i.e., $x \equiv y$ for every $x,y \in A$.
\end{definition}

Observe that $x \curlywedge y$ and $x \curlyvee y$ are stable, for every $x,y \in L$.

\begin{proposition}\label{stable-wedge-02}
    Let $A,B,C\subseteq L$ be stable sets, and let $\#,\#' \in\{\curlywedge, \curlyvee\}$. Then, $A \# B$ is stable and, for all $a\in A$, $b\in B$ and $c\in C$:
    \begin{enumerate}
        \item $A \# B= a \# b$.

        \item $A\#(B\#' C)=a\#(b\#' c)$ \ and \ $(A\# B)\#' C=(a\#b)\#' c$.

        \item $A\#(B\# C)=a\#(b\# c)=(a\#b)\# c=(A\# B)\# C$.

        \item $A \# B \preceq C$ iff $a \# b \preceq c$ for some $a \in A$, $b \in B$ and $c \in C$.

        \item $C \preceq A \# B$ iff $c \preceq a \# b$ for some $a \in A$, $b \in B$ and $c \in C$.
    \end{enumerate}
\end{proposition}

\begin{definition}  Let $\mathsf P=\langle P, \preceq, \curlywedge, \curlyvee\rangle$ be a hyperlattice. The sets  $\mathsf{Min}(P)$ and $\mathsf{Max}(P)$ of minima and maxima elements of $P$ will be denoted by $\bot$ and $\top$, respectively.
\end{definition}

In his Master's dissertation from 1971, A.M. Sette introduced an interesting notion of implicative hyperlattices, based on Morgado hyperlattices:

\begin{definition} [Sette implicative hyperlattices, {\cite[Definition~2.3]{sette1971algebras}}] \label{def:s-ilattices} A {\em Sette implicative hyperlattice} (or an {\em IHL}) is a hyperalgebra $\mathsf L=\langle L,\curlywedge,\curlyvee,\multimap  \rangle$ such that the reduct $\langle L,\curlywedge,\curlyvee\rangle$  is a hyperlattice and the hyperoperator $\multimap$  satisfies the following properties, for every $x,y,z, z' \in L$:
\begin{description}
    \item[(I1)] $z \in x \multimap y$ implies that $x \curlywedge z \preceq y$;

    \item[(I2)] $x \curlywedge z \preceq y$ implies that   $z \preceq x \multimap y$;

    \item[(I3)] $z \equiv z'$ and $z \in x \multimap y$ implies that $z' \in x \multimap y$.
\end{description}
\end{definition}

\noindent
It is possible to give a useful characterization of IHLs.

\begin{definition}  Let $\mathsf P$ be a hyperlattice, and let $x,y \in P$. The set $\mathsf R(x,y)$ is given by $\mathsf R(x,y)=\{z \in P \ : \ x \curlywedge z \preceq y\}$.
\end{definition}

\begin{proposition} [{\cite[Proposition~10]{coniglio2020non}}] \label{prop:char:IHL} Let $\mathsf L=\langle L,\curlywedge,\curlyvee,\multimap  \rangle$ be a  hyperalgebra such that $\langle L,\curlywedge,\curlyvee\rangle$  is a hyperlattice. Then, $\mathsf L$ is an  IHL iff $x \multimap y = \mathsf{Max}(\mathsf R(x,y))$, for every $x,y \in L$.
\end{proposition}

It follows that $x \multimap y$ is stable, for every $x,y \in L$. Moreover, the following holds (see~\cite[Section~3]{coniglio2020non}):

\begin{proposition} \label{prop-SIHL}
Let $\mathsf L$ be an IHL, and let $x,y,z \in L$. Then:
\begin{enumerate}
    \item If $A,B\subseteq L$ are stable then $A\multimap B$ is stable. Hence  $A\multimap B=a\multimap b$  for all $a\in A$ and all $b\in B$.
    
    \item $A  \preceq B$ iff   $a \multimap b = \top$ for every  $a \in A$ and $ b \in B$, iff $A \multimap B = \top$. In particular, if $A,B$ are stable then $A\preceq B$ iff $a \multimap b = \top$ for some $a\in A$ and some $b\in B$ iff $a \preceq b$ for some $a\in A$ and some $b\in B$.
    
\end{enumerate}
\end{proposition}

\section{The logic $C_\omega$}\label{sec:Cw}

In his groundbreaking Habilitation thesis~\cite{da1993sistemas}, Newton da Costa introduced the hierarchy of paraconsistent logics $C_n$ (for $1 \leq n \leq \omega$). The system $C_\omega$ is a kind of `syntactic limit' of the hierarchy, as it contains exactly all the axioms belonging simultaneously to all the calculi $C_n$, for $1 \leq n < \omega$. However, $C_\omega$ is not the deductive limit of these calculi, see~\cite{car:mar:99}. Among other features, $C_\omega$ is not finitely trivializable (that is, it cannot define a bottom formula) and, different from the other calculi of the hierarchy, it does not validate the Peirce/Dummett law $\varphi \vee (\varphi \to \psi)$. Hence, it is an expansion ---   by adding a paraconsistent negation --- of  positive intuitionistic logic, instead of expanding positive classical logic (and, {\em a posteriori}, classical logic), as $C_n$ does  for $1 \leq n < \omega$.

\

\begin{definition} [Hilbert calculus for  $C_\omega$] \label{syCw} The Hilbert calculus for  $C_\omega$ is defined over the signature $\Sigma_\omega = \{\land,\vee,\to,\neg\}$  as follows:
\\[2mm]
{\bf Axiom schemas:}
\begin{multicols}{2}
    \begin{description}
    \item[(AX1)] $\alpha  \to (\beta  \to \alpha )$
    \item[(AX2)] $(\alpha  \to (\beta  \to \gamma )) \to ((\alpha  \to \beta ) \to (\alpha  \to \gamma ))$
    \item[(AX3)] $\alpha  \to (\beta  \to (\alpha  \land \beta ))$
    \item[(AX4)] $(\alpha  \land \beta ) \to \alpha$
    \item[(AX5)] $(\alpha  \land \beta ) \to \beta$
    \item[(AX6)] $\alpha  \to (\alpha  \lor \beta )$
    \item[(AX7)] $\beta  \to (\alpha  \lor \beta )$
    \item[(AX8)] $(\alpha  \to \gamma ) \to ((\beta  \to \gamma ) \to	((\alpha  \lor \beta ) \to \gamma ))$
    \item[(EM)] $\alpha  \lor \neg \alpha$
    \item[(cf)] $\neg\neg \alpha  \to  \alpha$
\end{description}
\end{multicols}

\textbf{Inference rule:}
\begin{description}
    \item[(MP)] $\inferrule{\alpha\quad\alpha\rightarrow\beta}{\beta}$ 
\end{description}
\end{definition}

It is worth noting that (AX1)-(AX8) plus (MP) constitute the standard Hilbert calculus for positive intuitionistic logic, which is semantically characterized by the class of implicative lattices.

In~\cite[Chapter~2]{sette1971algebras}, Sette introduced a class of hyperalgebras for $C_\omega$, which are in correspondence with da Costa algebras for $C_\omega$ proposed in~\cite{daC:set:69}. Thus, they  constitute a suitable semantics for $C_\omega$. A slightly more general definition was considered in~\cite{con:gol:rob:2025}. Observe that $\top\neq\emptyset$ in any IHL: indeed, for every $x \in L$, $x \multimap x$ is a non-empty set contained in $\top$.

\begin{definition} [Sette hyperalgebras  for $C_\omega$, {\cite[Definition~16]{con:gol:rob:2025}}] \label{def:SHCw}
A {\em Sette hyperalgebra for  $C_\omega$} (or {\em hyper  $C_\omega$ algebra}, or simply a HC$_\omega$A) is a hyperalgebra $\mathsf H=\langle H,\curlywedge,\curlyvee,\multimap, \div  \rangle$ over $\Sigma_\omega$ such that the reduct~\mbox{$\langle H,\curlywedge,\curlyvee,\multimap\rangle$}  is an IHL and the hyperoperator $\div$  satisfies the following properties, for every $x, y, w \in H$:

\begin{description}
    \item[(H1)] $y \in \div x$ and $w \in x \curlyvee y$ implies that  $w \in \top$;
    \item[(H2)] $y \in \div x$ and $w\in \div y$ implies that  $w \preceq x$.
\end{description}
\end{definition}

\noindent It is immediate to see that conditions (H1) and (H2) can be written in a concise way as follows:

\begin{description}
    \item[(H1')] $x \curlyvee \div x \equiv \top$;
    \item[(H2')] $\div \div x \preceq x$,
\end{description}

\noindent
for every $x \in H$. 

\begin{definition} [HC$_\omega$A semantics] \label{def-sem-SHCw} Let $\mathsf H$ be a HC$_\omega$A, and let $\Gamma \cup \{\varphi\}$ be a set of formulas over $\Sigma_\omega$.
\begin{enumerate}
    \item The Nmatrix associated to $\mathsf H$ is $\mathcal M_{\mathsf H}=\langle \mathsf H, \top\rangle$.

    \item We say that $\varphi$ is a semantical consequence of $\Gamma$ w.r.t. $\mathsf H$ if $\Gamma \models_{\mathcal M_{\mathsf H}} \varphi$.

    \item Let $\mathbb{HC}_\omega$ be the class of HC$_\omega$As. Then, $\varphi$ is a semantical consequence of $\Gamma$ w.r.t.  HC$_\omega$As, denoted by   $\Gamma \models_{\mathbb{HC}_\omega} \varphi$, if $\Gamma \models_{\mathcal M_{\mathsf H}} \varphi$ for every $\mathsf H \in \mathbb{HC}_\omega$.
\end{enumerate}
\end{definition}

\section{Hyper Swap structures for $C_\omega$}\label{sec:swap-str}

In our previous paper~\cite{con:gol:rob:2025} we proved that $C_\omega$ can be semantically characterized by means of swap structures, which are Sette hyperalgebras for $C_\omega$ defined over pairs of elements of a given implicative lattice  $\mathsf L$. Given  an element $z \in L \times L$, the first and second components of $z$ will be denoted, respectively, by $z_1$ and $z_2$. That is, $z=(z_1,z_2)$.  As usual in the context of swap structures, these pairs are called {\em snapshots}. The intuitive meaning of the snapshots is that each coordinate represents, respectively, the values associated with formulas $\varphi$ and $\neg\varphi$ in the underlying implicative lattices. For convenience, we briefly recall it below::

\begin{definition} [Swap structures for $C_\omega$] \label{def-swap-Cw} Let $\mathsf L=\langle L, \wedge,\vee,\to\rangle$ be an implicative lattice. Let $S_{\mathsf L}=\{z \in L \times L \ : \ z_1 \vee z_2=1 \}$. The swap structure for $C_\omega$ over $\mathsf L$ is the hyperalgebra $\Sw_0(\mathsf L)=\langle S_{\mathsf L}, \breve{\land},\breve{\vee},\breve{\to},\breve{\neg}\rangle$ over the signature $\Sigma_\omega$ such that the hyperoperators are defined as follows:\\

$\begin{array}{lll}
z\breve{\land} w = \{u\in S_{\mathsf L} \ : \ u_1=z_1\wedge w_1 \} & \hspace{1cm} & z\breve{\to} w = \{u\in S_{\mathsf L} \ : \ u_1=z_1\to w_1 \}\\[1mm]
z\breve{\vee} w = \{u\in S_{\mathsf L} \ : \ u_1=z_1\vee w_1 \} & \hspace{1cm} & \breve{\neg} z= \{u\in S_{\mathsf L} \ : \ u_1=z_2 \mbox{ and }  u_2 \leq z_1 \} \\[1mm]
\end{array}$
\end{definition}

\noindent The Nmatrix associated to  $\Sw_0(\mathsf L)$ is $\mathcal{M}_0(\mathsf L)=\langle \Sw_0(\mathsf L),D_{\mathsf L} \rangle$ where the set of designated truth-values is $D_{\mathsf L}=\{z \in S_{\mathsf L} \ : \ z_1=1\}$. Let $\models_{C_\omega}^{SW}$ be the consequence relation generated by the class of Nmatrices of the form $\mathcal{M}_0(\mathsf L)$. Then, the following result holds (see~\cite[Theorem~1]{con:gol:rob:2025}):

\begin{theorem} [Soundness and completeness of  $C_\omega$ w.r.t. hyperstructures semantics, version~1] \label{Sound-comple0} \ \\
Let $\Gamma \cup \{\varphi\}$ be a set of formulas over $\Sigma_\omega$. The following assertions are equivalent:
\begin{enumerate}
    \item $\Gamma \vdash_{C_\omega} \varphi$;
    \item $\Gamma \models_{\mathbb{HC}_\omega} \varphi$;
    \item $\Gamma \models_{C_\omega}^{SW} \varphi$.
\end{enumerate}
\end{theorem}

\begin{remark} \label{remark:hyperswap}
Swap structures generalize, to the hyperalgebraic setting, the well-known {\em twist structures} technique mentioned in the Introduction. Instead of looking, by means of a Kalman functor,  for an equivalence  between the `base' category of implicative lattices and a subcategory of the induced category of Sette hyperalgebras for $C_\omega$ (with swap structures as `representative' objects), in this section we will generalize the standard construction of  swap structures (over a certain class of algebras) to the novel notion of {\em hyper swap structures}. They are swap structures defined over hyperalgebras instead of algebras. In this specific case, we will consider Sette  implicative hyperlattices instead of standard implicative lattices. The reasons for considering hyper swap structures will be clear in Section~\ref{sect:equiv01}. Indeed, by using a subcategory of  {\em enriched hyper C$_\omega$ algebras} (see Definition~\ref{def:EHCw}), we will obtain an equivalence between the category of Sette  implicative hyperlattices and enriched hyper C$_\omega$ algebras, where the hyper swap structures for $C_\omega$ will play the role of `representative' objects. This  takes the equivalences that can be established between categories of different classes of algebras by means of twist structures into the hyperalgebraic  context.    
\end{remark}

From now on, given an IHL $\mathsf L$ and an element $z \in L \times L$, the first and second components of $z$ will be denoted, as in the case of swap structures, by $z_1$ and $z_2$, respectively. These pairs will also be referred to as {\em snapshots}, and they represent, in the hyper swap structures to be defined below, the values (in this case, in a given IHL)  associated to formulas $\varphi$ and $\neg\varphi$. By generalizing the swap structures for $C_\omega$, moving from implicative lattices to implicative hyperlattices, we arrive at the following notion:

\begin{definition} [Hyper Swap structures for $C_\omega$] \label{def-hswap-Cw} Let $\mathsf L=\langle L, \curlywedge,\curlyvee,\multimap\rangle$ be an IHL. Let 
$$S^{C_\omega}_{\mathsf L}=\{z \in L \times L \ : \ z_1 \curlyvee z_2\equiv\top \}.$$
The hyper swap structure for $C_\omega$ over $\mathsf L$ is the hyperalgebra $\mathsf S(\mathsf L)=\langle S^{C_\omega}_{\mathsf L}, \curlywedge,\curlyvee,\multimap,\div\rangle$ over the signature $\Sigma_\omega$ such that the hyperoperators are defined as follows:\footnote{By simplicity, the hyperoperators and the induced preorder in $\mathsf S(\mathsf L)$ will be denoted by using the same symbols as in $\mathsf L$. The context will avoid any confusion.}
\begin{align*}
    z\curlywedge w&:=\{u\in S^{C_\omega}_{\mathsf L} \ : \ u_1\in z_1\curlywedge w_1 \}\\
    z\curlyvee w&:=\{u\in S^{C_\omega}_{\mathsf L} \ : \ u_1\in z_1\curlyvee w_1 \}\\
    z\multimap w&:=\{u\in S^{C_\omega}_{\mathsf L} \ : \ u_1\in z_1\multimap w_1 \}\\
    \div z&:=\{u\in S^{C_\omega}_{\mathsf L} \ : \ u_1=z_2 \mbox{ and }  u_2 \preceq z_1 \}
\end{align*}
\end{definition}

\noindent Following the usual definitions for swap structures, each hyper swap structure can be naturally associated with an Nmatrix:

\begin{definition} \label{def-Nmatrix-Cw}
 Let $\mathsf L$ be an IHL. The Nmatrix associated to $\mathsf S(\mathsf L)$ is $\mathcal{M}(\mathsf L)=\langle \mathsf S(\mathsf L),D^{C_\omega}_{\mathsf L} \rangle$ where the set of designated truth-values is $D^{C_\omega}_{\mathsf L}=\{z \in S^{C_\omega}_{\mathsf L} \ : \ z_1\in\top\}$.
\end{definition}

\begin{proposition} \label{swap-SHCw}
Let $\mathsf L$ be  an IHL, and let $\mathsf S(\mathsf L)$ be the hyper swap structure for $C_\omega$ over $\mathsf L$. Then:
\begin{enumerate}
    \item The relation $z \preceq w$  in  $\mathsf S(\mathsf L)$  iff  $z_1 \preceq w_1$  in $\mathsf L$
    defines a  preorder such that $\mathsf S(\mathsf L)$ is an hyperlattice where, for every $z,w \in  S^{C_\omega}_{\mathsf L}$, $z\curlywedge w$ and $z\curlyvee w$ are the infimoid and the supremoid of $z$ and $w$, respectively.
    Moreover, $z \equiv w$ in  $\mathsf S(\mathsf L)$ iff $z_1 \equiv w_1$ in $\mathsf L$. 
    \item $\mathsf S(\mathsf L)$ is a HC$_\omega$A.
    Moreover, $D^{C_\omega}_{\mathsf L}=\top$.
    \item $\mathcal{M}(\mathsf L)= \mathcal M_{\mathsf S(\mathsf L)}$.
\end{enumerate}
\end{proposition}
\begin{proof} 1. Clearly, $\preceq$ is a preorder in $\mathsf S(\mathsf L)$. Observe that $\top=\mathsf{Max}(L) \neq \emptyset$: for instance, $\emptyset \neq u\multimap u\subseteq\top$, for any $u \in L$. \\[1mm]
{\bf Fact 1:} If $a \in L$ and $b \in \top$ then $(a,b) \in S^{C_\omega}_{\mathsf L}$.\\[1mm]
Indeed, if $b \in \top$ and $c \in L$ then  $c \preceq b \preceq a \curlyvee b$. This means that $a \curlyvee b \equiv \top$, and so $(a,b) \in S^{C_\omega}_{\mathsf L}$. This proves  {\bf Fact~1}.\\[1mm]
As a direct consequence of {\bf Fact 1}, and given that $z_1 \curlywedge w_1 \neq \emptyset \neq z_1 \curlyvee w_1$, it follows that   $z \curlywedge w \neq \emptyset \neq z \curlyvee w$ for every $z,w \in S^{C_\omega}_{\mathsf L}$.
Now, let us prove that $z\curlyvee w$ is the supremoid in  $\mathsf S(\mathsf L)$ of $z$ and $w$. Thus, let $u \in z\curlyvee w$. By definition, $u_1 \in z_1 \curlyvee w_1$ and so $z_1,w_1 \preceq u_1$. Then, $z,w \preceq u$ and so $u \in \mathsf{Ub}(\{z,w\})$. Let $x \in \mathsf{Ub}(\{z,w\})$. Then, $z,w \preceq x$ and so $z_1,w_1 \preceq x_1$, therefore $z_1 \curlyvee w_1 \preceq x_1$. From this, $u_1 \preceq x_1$, hence $u \preceq x$. This shows that $u \in \mathsf{Min}(\mathsf{Ub}(\{z,w\}))$, that is, $z\curlyvee w \subseteq \mathsf{Min}(\mathsf{Ub}(\{z,w\}))$. Now, let  $u \in \mathsf{Min}(\mathsf{Ub}(\{z,w\}))$. Since $z,w \preceq u$ then $z_1,w_1 \preceq u_1$, that is, $u_1 \in \mathsf{Ub}(\{z_1,w_1\})$. Let $a \in  \mathsf{Ub}(\{z_1,w_1\})$. By {\bf Fact~1}, $(a,b) \in  S^{C_\omega}_{\mathsf L}$ for any $b \in \top$ such that $z,w \preceq (a,b)$. Thus, $u \preceq (a,b)$ which implies that $u_1 \preceq a$. Then, $u_1 \in \mathsf{Min}(\mathsf{Ub}(\{z_1,w_1\})) = z_1 \curlyvee w_1$. That is, $u \in  z \curlyvee w$ and so $z\curlyvee w = \mathsf{Min}(\mathsf{Ub}(\{z,w\}))$. The proof that $z\curlywedge w = \mathsf{Max}(\mathsf{Lb}(\{z,w\}))$ is analogous. This shows that  $\mathsf S(\mathsf L)$ is an hyperlattice where  $z \equiv w$ in  $\mathsf S(\mathsf L)$ iff $z_1 \equiv w_1$ in $\mathsf L$. \\[1mm]
2. By {\bf Fact~1} above, and given that $z_1 \multimap w_1 \neq \emptyset$, it follows that   $z \multimap w \neq \emptyset$ for every $z,w \in S^{C_\omega}_{\mathsf L}$. \\[1mm]
{\bf Fact 2:} For every $x,z,w  \in S^{C_\omega}_{\mathsf L}$ it holds: $z \curlywedge x \preceq w$ iff  $z_1 \curlywedge x_1 \preceq w_1$.\\[1mm]
Indeed, suppose that  $z \curlywedge x \preceq w$, and let $a \in z_1 \curlywedge x_1$. Let $b \in \top$. By {\bf Fact 1}, $(a,b) \in  S^{C_\omega}_{\mathsf L}$ such that $(a,b) \in z \curlywedge x$. By hypothesis, $(a,b) \preceq w$ and so $a \preceq w_1$. That is, $z_1 \curlywedge x_1 \preceq w_1$. Conversely, suppose that  $z_1 \curlywedge x_1 \preceq w_1$ and let  $u \in z \curlywedge x$. Then,   $u_1 \in z_1 \curlywedge x_1$, which implies that $u_1 \preceq w_1$. This means that $u \preceq w$, therefore  $z \curlywedge x \preceq w$. This proves  {\bf Fact~2}.\\[1mm] 
Now, given $z,w \in S^{C_\omega}_{\mathsf L}$, let $u \in z \multimap w$. Then,  $u_1 \in z_1 \multimap w_1$, hence  $z_1 \curlywedge u_1 \preceq w_1$. By {\bf Fact~2}, $z \curlywedge u \preceq w$ and so  $u \in \mathsf R(z,w)$. Now, suppose that  $z \curlywedge x \preceq w$. By {\bf Fact~2} once again,  $z_1 \curlywedge x_1 \preceq w_1$, which implies that $x_1 \preceq  z_1 \multimap w_1$. From this, $x_1 \preceq u_1$, hence $x \preceq u$. That is, $u \in \mathsf{Max}(\mathsf R(z,w))$, proving that  $z \multimap w \subseteq \mathsf{Max}(\mathsf R(z,w))$. Conversely, let $u \in \mathsf{Max}(\mathsf R(z,w))$. Then,  $z \curlywedge u \preceq w$ and so  $z_1 \curlywedge u_1 \preceq w_1$, by {\bf Fact~2}. Let $a \in L$ such that  $z_1 \curlywedge a \preceq w_1$. For any $b \in \top$ it follows, by {\bf Fact~1} and {\bf Fact~2}, that  $z \curlywedge (a,b) \preceq w$, hence $(a,b) \preceq u$. From this, $a \preceq u_1$, showing that $u_1 \in \mathsf{Max}(\{a \in L \ : \  z_1 \curlywedge a \preceq w_1\}) = z_1 \multimap w_1$, by Proposition~\ref{prop:char:IHL}. That is, $u \in z \multimap w$. This shows that $z \multimap w =\mathsf{Max}(\mathsf R(z,w))$ and so $\mathsf S(\mathsf L)$ is an IHL, by Proposition~\ref{prop:char:IHL}.

    Finally, let $z\in S^{C_\omega}_{\mathsf L}$. We have that
    \begin{align*}
        z_1\curlyvee z_2&\equiv\top\nonumber\\
        \div z&=\{w\in  S^{C_\omega}_{\mathsf L} \ : \ w_1=z_2\mbox{ and }w_2\preceq z_1\}\nonumber\\
        \div\div z&=\{u\in  S^{C_\omega}_{\mathsf L} \ : \ u_1\preceq z_1\mbox{ and }u_2\preceq z_2\}
    \end{align*}
    Since for $z,w\in  S^{C_\omega}_{\mathsf L}$, $z\preceq_{\mathsf S(\mathsf L)}w$ iff $z_1\preceq_{\mathsf L}w_1$, these expressions witness the validity of (H1) and (H2) for $\mathsf S(\mathsf L)$ (recall Definition~\ref{def:SHCw}). That is, $\mathsf S(\mathsf L)$ is a HC$_\omega$A. Clearly,  $D^{C_\omega}_{\mathsf L}= \mathsf{Max}(S^{C_\omega}_{\mathsf L})$.\\[1mm]
    3. It follows from the definitions and by the fact that $D^{C_\omega}_{\mathsf L}= \mathsf{Max}(S^{C_\omega}_{\mathsf L})$.
\end{proof}

\begin{definition} [Hyper Swap structures semantics for  $C_\omega$] \label{def-sem-swap-Cw} Let $\Gamma \cup \{\varphi\}$ be a set of formulas over $\Sigma_\omega$. Then, $\varphi$ is a semantical consequence of $\Gamma$ w.r.t. hyper swap structures, denoted by $\Gamma \models_{C_\omega}^{HSW} \varphi$, whenever $\Gamma \models_{\mathcal M(\mathsf L)} \varphi$ for every implicative hyper lattice $\mathsf L$.
\end{definition}

\noindent In order to prove  soundness and completeness of $C_\omega$ w.r.t. hyper swap structures semantics, we recall here some well-known  notions and results concerning (Tarskian) logics.

Given a Tarskian and finitary logic {\bf L} and a set of formulas $\Delta \cup \{ \varphi\}$ of {\bf L}, the set $\Delta$ is said to be {\em $\varphi$-saturated in} {\bf L} if the following holds:~(i)~$\Delta \nvdash_{\bf L} \varphi$; and~(ii)~if $\psi \notin \Delta$ then $\Delta,\psi \vdash_{\bf L}\varphi$.

It follows immediately that any  $\varphi$-saturated in a  Tarskian logic is deductively closed, i.e.: $\psi \in \Delta$ iff $\Delta \vdash_{\bf L} \psi$.

By a classical result proven by Lindenbaum and \L o\'s,  if $\Gamma \cup \{ \varphi\}$ is a set of formulas of a Tarskian and finitary logic {\bf L} such that $\Gamma \nvdash_{\bf L} \varphi$, then there exists a $\varphi$-saturated set $\Delta$  such that $\Gamma \subseteq \Delta$.\footnote{For a proof of this result see, for instance, \cite[Theorem~22.2]{wojcicki1984lectures} or~\cite[Theorem~2.2.6]{carnielli2016paraconsistent}.} Since $C_\omega$ is a Tarskian and finitary logic, Lindenbaum-\L o\'s Theorem holds for it.

\begin{theorem} [Soundness and completeness of  $C_\omega$ w.r.t. hyperstructures semantics, version~2] \label{Sound-comple} \ \\
Let $\Gamma \cup \{\varphi\}$ be a set of formulas over $\Sigma_\omega$. The following assertions are equivalent:
\begin{enumerate}
    \item $\Gamma \vdash_{C_\omega} \varphi$;
    \item $\Gamma \models_{\mathbb{HC}_\omega} \varphi$;
    \item $\Gamma \models_{C_\omega}^{HSW} \varphi$.
\end{enumerate}
\end{theorem}
\begin{proof}
$(1) \Rightarrow (2)$ (Soundness of   $C_\omega$ w.r.t. HC$_\omega$As). This was proven in~\cite[Theorem~1]{con:gol:rob:2025}.\\[1mm]
$(2) \Rightarrow (3)$. It follows by Proposition~\ref{swap-SHCw}, items~(2) and~(3).\\[1mm]
$(3) \Rightarrow (1)$ (Completeness of   $C_\omega$ w.r.t. hyper swap structures semantics). Suppose that $\Gamma \nvdash _{C_\omega} \varphi$. Then, by Lindenbaum-\L o\'s result mentioned above, there exists a $\varphi$-saturated set $\Delta$ in $C_\omega$ such that $\Gamma \subseteq \Delta$. Now, define  a relation $\preceq_\Delta$ over $For(\Sigma_\omega)$ as follows: $\alpha \preceq_\Delta \beta$ iff $\Delta \vdash _{C_\omega} \alpha \to \beta$. Clearly, $\preceq_\Delta$ is a preorder, given that $C_\omega$  contains positive intuitionistic logic $IPL^+$, hence $\vdash_{C_\omega} \alpha \to \alpha$, and  $\alpha \to \beta, \, \beta \to \gamma \vdash_{C_\omega} \alpha \to \gamma$. Observe that, with this preorder, $\alpha \equiv \beta$ iff  $\Delta \vdash _{C_\omega} \alpha \leftrightarrow \beta$, where $\alpha \leftrightarrow \beta$ is an abbreviation for  $(\alpha \to \beta) \land (\beta \to \alpha)$. Using again the fact that $C_\omega$ is an axiomatic extension of $IPL^+$, it is easy to prove that $\langle For(\Sigma_\omega), \preceq_\Delta \rangle$ is an hyperlattice such that $\alpha \curlywedge \beta =\{\gamma \ : \ \Delta \vdash _{C_\omega} \gamma \leftrightarrow (\alpha \land\beta)\}$, and  $\alpha \curlyvee \beta =\{\gamma \ : \ \Delta \vdash _{C_\omega} \gamma \leftrightarrow (\alpha \vee\beta)\}$. Moreover, it is an IHL (that we will call $\mathsf{L}_\Delta$)  such that $\alpha \multimap \beta =\{\gamma \ : \ \Delta \vdash _{C_\omega} \gamma \leftrightarrow (\alpha \to\beta)\}$. Observe that $\top=\{\gamma \ : \ \Delta \vdash _{C_\omega} \gamma\}=\Delta$.

Let $\mathsf S(\mathsf{L}_\Delta)$ be the hyper swap structure for $C_\omega$ over $\mathsf{L}_\Delta$, with domain $S^{C_\omega}_{\mathsf{L}_\Delta}$, as in Definition~\ref{def-hswap-Cw}, and let $\mathcal{M}(\mathsf{L}_\Delta)$ be the associated Nmatrix (see Definition~\ref{def-Nmatrix-Cw}). Notice that $S^{C_\omega}_{\mathsf{L}_\Delta}=\{(\alpha,\beta) \ : \ \Delta \vdash _{C_\omega} \alpha \vee \beta\}$ and $D^{C_\omega}_{\mathsf{L}_\Delta}=\{(\alpha,\beta) \ : \ \Delta \vdash _{C_\omega} \alpha\}$ . Let $v_\Delta:For(\Sigma_\omega) \to S^{C_\omega}_{\mathsf{L}_\Delta}$ be the canonical map given by $v_\Delta(\alpha)= (\alpha,\neg \alpha)$, for every $\alpha$. It is easy to see that $v_\Delta$ is a valuation over $\mathcal{M}(\mathsf{L}_\Delta)$. Indeed, 
$$v_\Delta(\alpha\land \beta)= (\alpha\land\beta,\neg (\alpha\land \beta)) \in \{(\gamma, \delta) \in S^{C_\omega}_{\mathsf{L}_\Delta} \ : \ \gamma \in \alpha \curlywedge \beta\}= v_\Delta(\alpha) \curlywedge v_\Delta(\beta),$$
given that $\alpha \land \beta \in \alpha \curlywedge \beta=\{\gamma \ : \ \Delta \vdash _{C_\omega} \gamma \leftrightarrow (\alpha \land\beta)\}$. Analogously we prove that 
$$v_\Delta(\alpha\vee \beta) \in v_\Delta(\alpha) \curlyvee v_\Delta(\beta) \ \ \mbox{ and } \ \ v_\Delta(\alpha\to \beta) \in v_\Delta(\alpha) \multimap v_\Delta(\beta).$$
Finally, by axiom (cf) it is immediate to see that
$$v_\Delta(\neg\alpha)= (\neg\alpha,\neg\neg \alpha) \in \{(\gamma, \delta) \in S^{C_\omega}_{\mathsf{L}_\Delta} \ : \ \gamma = \neg\alpha \ \mbox{ and } \ \Delta  \vdash_{C_\omega} \delta \to \alpha\}= \div v_\Delta(\alpha).$$
Moreover, $v_\Delta(\alpha)\in D^{C_\omega}_{\mathsf{L}_\Delta}$ iff $\Delta \vdash_{C_\omega} \alpha$. From this, $v_\Delta(\alpha)\in D^{C_\omega}_{\mathsf{L}_\Delta}$ for every $\alpha \in \Gamma$, while $v_\Delta(\varphi)\notin D^{C_\omega}_{\mathsf{L}_\Delta}$, given that $\Delta \nvdash _{C_\omega} \varphi$. This shows that $\Gamma \not\models_{\mathcal{M}(\mathsf{L}_\Delta)} \varphi$ and so $\Gamma \not\models_{C_\omega}^{HSW} \varphi$.

This completes the proof.
\end{proof}

\section{An Equivalence of Categories between IHL and EHC$_\omega$A}\label{sect:equiv01}

Theorem~\ref{Sound-comple} is interesting by itself, as it provides another class of hyperstructures that characterize the logic  $C_\omega$. In this section, it will be shown that, more than this, hyper swap structures are crucial in order to adapt the Kalman functor to the non-deterministic context, as discussed in Remark~\ref{remark:hyperswap}.

\begin{definition}
    The category \textbf{IHL} is the one where the objects are IHLs and the morphisms are just the usual morphisms of hyperalgebras. In other words, given $\mathsf L_1,\mathsf L_2\in\mbox{IHL}$, a function $f:L_1\rightarrow L_2$ is a morphism from $\mathsf L_1$ to $\mathsf L_2$ if for all $x,y,z\in L_1$ we have the following:
    \begin{enumerate}
        \item $z\in x\curlywedge y$ implies $f(z)\in f(x)\curlywedge f(y)$;
        \item $z\in x\curlyvee y$ implies $f(z)\in f(x)\curlyvee f(y)$;
        \item $z\in x\multimap y$ implies $f(z)\in f(x)\multimap f(y)$
    \end{enumerate}
\end{definition}

\begin{lemma}\label{morphism-lemma}
    Let $\mathsf L_1,\mathsf L_2$ be IHLs and $f:L_1\rightarrow L_2$ be a function. If $f$ is an IHL-morphism then for all $x,y\in L_1$, $x\preceq y$ implies $f(x)\preceq f(y)$.
\end{lemma}
\begin{proof} By Definition~\ref{def:m-hlattices}, $x\curlywedge y=\mathsf{Max}(\mathsf{Lb}(x,y))$ for any $x,y \in L$, for any m-hyperlattice $\mathsf{L}$. Now, suppose that $f$ is an IHL-morphism and $x,y\in L_1$ such that $x\preceq y$. Then, $x \in x\curlywedge y$ and so  $f(x)\in f(x)\curlywedge f(y)$. Using once again the definition of infimoid, it follows that  $f(x)\preceq f(y)$. 
\end{proof}

\begin{definition}
    The category \textbf{HC$_\omega$A} is the one where the objects are HC$_\omega$As and the morphism are just the usual morphisms of hyperalgebras. In other words, given $\mathsf H_1,\mathsf H_2\in\mbox{HC}_\omega\mbox{A}$, a function $f:H_1\rightarrow H_2$ is a morphism from $\mathsf H_1$ to $\mathsf H_2$ if for all $x,y,z\in A_1$ we have the following:
    \begin{enumerate}
        \item $z\in x\curlywedge y$ implies $f(z)\in f(x)\curlywedge f(y)$;
        \item $z\in x\curlyvee y$ implies $f(z)\in f(x)\curlyvee f(y)$;
        \item $z\in x\multimap y$ implies $f(z)\in f(x)\multimap f(y)$;
        \item $z\in\div x$ implies $f(z)\in\div f(x)$.
    \end{enumerate}
\end{definition}



\begin{theorem}
    The hyper swap structure construction provides a Kalman functor $\mathsf S:\textbf{IHL}\rightarrow \textbf{HC}_\omega\textbf{A}$.
\end{theorem}
\begin{proof}
     We only need to extend $\mathsf S$ to  morphisms. Let $f:\mathsf L_1\rightarrow \mathsf L_2$ be an \textbf{IHL}-morphism. Define $\mathsf S(f):S^{C_\omega}_{\mathsf L_1}\rightarrow S^{C_\omega}_{\mathsf L_2}$ by the rule $\mathsf S(f)(z_1,z_2):=(f(z_1),f(z_2))$. Since $f$ is a morphism and $z_1\curlyvee z_2\equiv\top$, we get $f(z_1)\curlyvee f(z_2)\equiv\top$ and $\mathsf{S}(f)$ is well defined. Moreover, it is immediate to see that  $\mathsf S(f)$ is a morphism $\mathsf S(f):S(\mathsf L_1)\rightarrow S(\mathsf L_2)$ in $\textbf{HC}_\omega\textbf{A}$. After that, the fact that $\mathsf S(1_{\mathsf L_1})=1_{\mathsf S(\mathsf L_1)}$ and $\mathsf S(f\circ g)=\mathsf S(f)\circ \mathsf S(g)$ is a straightforward calculation involving the above definitions.
\end{proof}

\begin{remark}\label{enriched-rem}
    To obtain a functor in the reverse direction, we need to enrich the structure of a hyper $C_\omega$ algebra. The goal is to identify a class of hyper $C_\omega$ algebras which correspond (up to isomorphisms) to hyper swap structures for $C_\omega$. In order to do this, we will abstract the basic properties of hyper swap structures. Thus, let  $\mathsf S(\mathsf L)$ for a given IHL  $\mathsf L$. We know that the following three facts hold:
\begin{enumerate}
    \item For $z,z'\in S^{C_\omega}_{\mathsf L}$, $z_1=z'_1$ iff $z,z'\in\div w$ for some $w\in \mathsf S^{C_\omega}_{\mathsf L}$. To see this, let $z,z'\in S^{C_\omega}_{\mathsf L}$. Suppose that $z_1=z'_1$, and let $u \in z_2 \curlyvee z'_2$. Since $z_1 \curlyvee z_2 \equiv \top$ then, for every $x \in L$, $x \preceq z_1 \curlyvee z_2 \preceq z_1 \curlyvee (z_2 \curlyvee z'_2) = z_1 \curlyvee u = u \curlyvee z_1$. From this, $u \curlyvee z_1 \equiv \top$ and then $w:=(u,z_1) \in  S^{C_\omega}_{\mathsf L}$ is such that  $z,z'\in\div w$. Conversely, suppose that  $z,z'\in\div w$ for some $w\in S^{C_\omega}_{\mathsf L}$. Then $z_1=w_2=z'_1$.

    \item The relation $z\sim z'$ iff $z_1=z'_1$ is an equivalence relation in $S^{C_\omega}_{\mathsf L}$ that \emph{selects the first coordinate}. And by item~1, this relation has an alternative description: $z\sim z'$ iff $z,z'\in\div w$ for some $w\in S^{C_\omega}_{\mathsf L}$.

    \item The negation $\div$, together with $\sim$, determine $S^{C_\omega}_{\mathsf L}$ in the following sense: for $z,z'\in S^{C_\omega}_{\mathsf L}$, if $z\curlyvee z'\equiv\top$ (which occurs iff $z_1\curlyvee z'_1\equiv\top$) then there exists $w\in S^{C_\omega}_{\mathsf L}$ (namely $w=(z_1,z'_1)$) such that $z\sim w$ and $z'\sim \div w$. Moreover, $z=z'$ iff $z\sim z'$ and $\div z\sim\div z'$.\footnote{For $X,Y\subseteq\mathsf A$, $X\sim Y$ denotes that $x\sim y$ for all $x\in X$ and all $y\in Y$.}
\end{enumerate}
\end{remark}

\begin{definition}[Enriched Hyper C$_\omega$ Algebras] \label{def:EHCw}
    Let $\mathsf A=\langle A,\curlywedge,\curlyvee,\multimap,\div\rangle$ be a HC$_\omega$A. We say that $\mathsf A$ is an {\em enriched hyper C$_\omega$ algebra (EHC$_\omega$A)} if it satisfies the following additional axioms, for all $x,y,z\in A$:
    \begin{description}
        \item[E0 -] $x\in\div\div x$.
        \item[E1 -] $\div x$ is stable. In other words, if $y\in\div x$ and $z\in\div x$ then $y\equiv z$.
        \item[E2 -] The following relation $\sim$ is transitive (which implies, considering E0, that it is an equivalence relation):
        $$x\sim y\mbox{ iff there exists }z\mbox{ such that }x,y\in\div z.$$
        (Note that, by E1, $x\sim y$ implies $x\equiv y$.)
        \item[E3 -] If $x\curlyvee y\equiv\top$ then there exists $z$ such that $x\sim z$ and $y\sim\div z$.
        \item [E4 -] If $x\sim y$ and $\div x\sim\div y$ then $x=y$.
    \end{description}
    For $x\in A$ we denote $[x]:=\{y \ : \ x\sim y\}$ and
    $$\mathsf U(\mathsf A):=A/{\sim}=\{[x] \ : \ x\in A\}.$$
    The category \textbf{EHC$_\omega$A} is the one where the objects are EHC$_\omega$As and the morphism are the \textbf{HC$_\omega$A}-morphisms $f:\mathsf A_1\rightarrow \mathsf A_2$.
\end{definition}

\begin{remark}
Hyper $C_\omega$ algebras play a central role in our discussion: they allow us to define an appropriate subcategory of enriched hyperstructures equipped with a suitable congruence relation $\sim$, such that for any of such hyperstructures $\mathsf A$, the quotient $U(\mathsf A)=A/{\sim}$ yields an IHL, as it will be shown in Proposition~\ref{functor U- Cw}. 
Similar ideas already appear in \cite{odin:08} in the context of twist structures for Nelson logic N4. Specifically,  aiming to abstract the notion of twist structures for N4,  N4-lattices  (see~\cite[Definition~8.4.1]{odin:08}) establishes conditions analogous to those our EHC$_\omega$A  framework impose  to abstract the notion of hyper swap structures for $C_\omega$. For instance, condition 3 of Definition~8.4.1 in \cite{odin:08} is essentially the same as our Axiom E2, while  Axioms E3 and E4 transpose to our setting the content of condition 5 in that definition.
\end{remark}

We observe that enriched hyper $C_\omega$ algebras effectively abstract the notion of hyper swap structures for  $C_\omega$, as the following result shows:

\begin{proposition}\label{hyper-swap-is-enriched}
    Let $\mathsf L$ be an IHL. Then the hyper swap $\mathsf S(\mathsf L)$ is an enriched hyper $C_\omega$ algebra.
\end{proposition}
\begin{proof}
    Let $z\in S^{C_\omega}_{\mathsf L}$. Since
    \begin{align*}
        \div z&=\{w\in S^{C_\omega}_{\mathsf L} \ : \ w_1=z_2\mbox{ and }w_2\preceq z_1\}\\
        \div\div z&=\{u\in S^{C_\omega}_{\mathsf L} \ : \ u_1\preceq z_1\mbox{ and } u_2\preceq z_2\}
    \end{align*}
    we conclude that $\div z$ is stable and $z\in\div\div z$, proving E0 and E1. To prove E2, E3 and E4 just proceed as described in Remark \ref{enriched-rem}.
\end{proof}

From this, the logic naturally associated with enriched hyper $C_\omega$ algebras as a particular class of  HC$_\omega$As  is exactly $C_\omega$. To be more precise, for any $\Gamma \cup \{\varphi\} \subseteq For(\Sigma_\omega)$, define (recalling Definition~\ref{def-sem-SHCw}):  $\Gamma \models_{\mathbb{EHC}_\omega} \varphi$ iff  $\Gamma \models_{\mathcal M_{\mathsf H}} \varphi$ for every $\mathsf H$ which is an  EHC$_\omega$A.

\begin{corollary} \label{coro:logic-enriched}
Let $\Gamma \cup \{\varphi\}$ be a set of formulas over $\Sigma_\omega$. Then:
$\Gamma \vdash_{C_\omega} \varphi$ iff $\Gamma \models_{\mathbb{EHC}_\omega} \varphi$.
\end{corollary}
\begin{proof}
Suppose first that $\Gamma \vdash_{C_\omega} \varphi$. By   Theorem~\ref{Sound-comple}, $\Gamma \models_{\mathbb{HC}_\omega} \varphi$, therefore $\Gamma \models_{\mathbb{EHC}_\omega} \varphi$, given that any EHC$_\omega$A is in particular a HC$_\omega$A. Conversely, suppose that $\Gamma \models_{\mathbb{EHC}_\omega} \varphi$. In particular, $\Gamma \models_{C_\omega}^{HSW} \varphi$, by Proposition~\ref{hyper-swap-is-enriched}. This implies that $\Gamma \vdash_{C_\omega} \varphi$, by   Theorem~\ref{Sound-comple}.
\end{proof}

\begin{proposition} \label{functor U- Cw}
    Let $\mathsf A$ be an EHC$_\omega$A. Then $\langle\mathsf U(\mathsf A),\preceq_{\mathsf U(\mathsf A)}\rangle$ is an IHL (which will be also denoted by $\mathsf U(\mathsf A)$) with the relation $\preceq_{\mathsf U(\mathsf A)}$ defined by $[x]\preceq_{\mathsf U(\mathsf A)}[y]$ iff $x\preceq y$. Moreover, the assignment $\mathsf A\mapsto \mathsf U(\mathsf A)$ provides a functor $\mathsf U:\textbf{EHC}_\omega\textbf{A}\rightarrow \textbf{IHL}$.
\end{proposition}
\begin{proof}
    Note that, since $\div x$ is stable for all $x\in A$, $x\sim x'$ and $y\sim y'$ implies that $x\equiv x'$ and $y\equiv y'$. This implies that $x\preceq y$ iff $x'\preceq y'$, and $\preceq_{\mathsf U(\mathsf A)}$ is well-defined. To obtain the IHL structure for $\mathsf U(\mathsf A)$, just observe that
    \begin{align*}
        [x]\curlywedge_{\mathsf (\mathsf A)}[y]&:=\{[z] \ : \ z\in x\curlywedge y\}\\
        [x]\curlyvee_{\mathsf U(\mathsf A)}[y]&:=\{[z] \ : z\in x\curlyvee y\}\\
        \mathsf R_{\mathsf U(\mathsf A)}([x],[y])&:=\{[z] \ : \ z \in \mathsf R(x,y)\}
    \end{align*}
    are well-defined non-empty sets, which constitute the basis for an IHL structure over  $\langle\mathsf U(\mathsf A),\preceq_{\mathsf U(\mathsf A)}\rangle$. To finish the proof and to obtain a functor, let $f:\mathsf A_1\rightarrow \mathsf A_2$ be an \textbf{EHC$_\omega$A}-morphism and define $\mathsf U(f):\mathsf U(\mathsf A_1)\rightarrow \mathsf U(\mathsf A_2)$ given by the rule $\mathsf U(f)([x])=[f(x)]$. If $x\sim x'$ then $x,x'\in\div w$ for some $w\in A_1$, which implies that $f(x),f(x')\in\div f(w)$ and so $f(x)\sim f(x')$. Hence, $[f(x)]=[f(x')]$ and $\mathsf U(f)$ is well-defined.
    
    After that, the fact that $\mathsf U(1_{\mathsf A_1})=1_{\mathsf U(\mathsf A_1)}$ and $\mathsf U(f\circ g)=\mathsf U(f)\circ \mathsf U(g)$ is a straightforward calculation involving the above definitions.
\end{proof}

Now it is time to come back to the functors. Observe that, since the hyper swap structure $\mathsf S(\mathsf L)$ over an IHL $\mathsf L$ is an EHC$_\omega$A, the Kalman functor $\mathsf S:\textbf{IHL}\rightarrow \textbf{HC}_\omega\textbf{A}$ can be seen as a functor $\mathsf S:\textbf{IHL}\rightarrow \textbf{EHC}_\omega\textbf{A}$. 

\begin{theorem}\label{equivalence-01}
    For all $\mathsf L\in \textbf{IHL}$ there is an isomorphism $\Phi_{\mathsf L}:\mathsf L\rightarrow \mathsf U(\mathsf S(\mathsf L))$. Moreover, this provides a natural isomorphism of functors $\Phi:1_{\textbf{IHL}}\Rightarrow \mathsf U\circ \mathsf S$ given by $\mathsf L\mapsto\Phi_{\mathsf L}$.
\end{theorem}
\begin{proof}
    For all $x\in L$ there exist $y\in L$ such that $x\curlyvee y\equiv\top$ (for example, take $y\in x\multimap x$). Define $\Phi_{\mathsf L}(x):=[(x,y)]$ with $y\in L$ such that $x\curlyvee y\equiv\top$. The function  $\Phi_{\mathsf L}$ is well-defined: indeed, let $y,y'\in L$ such that $x\curlyvee y\equiv\top$ and $x\curlyvee y'\equiv\top$. Then, $(x,y),(x,y')$ are elements of $S^{C_\omega}_{\mathsf L}$ such that $(x,y)\sim(x,y')$ and so $[(x,y)]=[(x,y')]$. This proves that  $\Phi_{\mathsf L}$ is well-defined. 
    
    Now, let $x,x'\in L$ and $d\in x\curlywedge x'$. Also let $d'\in L$ such that $d\curlyvee d'\equiv\top$. Observe that for all $y,y'\in L$ with $x\curlyvee y\equiv\top$ and $x'\curlyvee y'\equiv\top$, we have $(d,d')\in(x,y)\curlywedge(x',y')$ in $\mathsf S(\mathsf L)$, which imply $[(d,d')]\in[(x,y)]\curlywedge[(x',y')]$ (in $\mathsf U(\mathsf S(\mathsf L))$). Then we have that $d\in x\curlywedge x'$ implies that $\Phi_L(d)\in \Phi_L(x)\curlywedge \Phi_L(x')$. Similarly we prove that $d\in x\curlyvee x'$ implies that $\Phi_L(d)\in \Phi_L(x)\curlyvee \Phi_L(x')$ and $d\in x\multimap x'$ implies $\Phi_L(d)\in \Phi_L(x)\multimap \Phi_L(x')$, proving that $\Phi_{\mathsf L}$ is an \textbf{IHL}-morphism. Note that this is a surjective morphism: if $[(x,y)]\in \mathsf U(\mathsf S(\mathsf L))$, then $[(x,y)]=\Phi_{\mathsf L}(x)$.

    Finally, suppose that $\Phi_{\mathsf L}(x)=\Phi_{\mathsf L}(x')$. This means that for some $y,y'\in L$ with $x\curlyvee y\equiv\top$ and $x'\curlyvee y'\equiv\top$ we have $[(x,y)]=[(x',y')]$. Then $(x,y)\sim(x',y')$, implying that $(x,y),(x',y')\in\div z$ for some $z\in S^{C_\omega}_{\mathsf L}$. In particular $x=x'$, showing that $\Phi_{\mathsf L}$ is injective.

    Therefore $\Phi_{\mathsf L}:\mathsf L\rightarrow \mathsf U(\mathsf S(\mathsf L))$ is an isomorphism which is natural in the sense that for an \textbf{IHL}-morphism $f:\mathsf L_1\rightarrow \mathsf L_2$, the following diagram commutes: 
    $$\xymatrix@!=4.5pc{
    \mathsf L_1\ar[r]^{\Phi_{\mathsf L_1}}\ar[d]_{f} & \mathsf U(\mathsf S(\mathsf L_1))\ar[d]^{\mathsf U(\mathsf S(f))} \\
    \mathsf L_2\ar[r]_{\Phi_{\mathsf L_2}} & \mathsf U(\mathsf S(\mathsf L_2))
    }$$
    In fact, for all $x\in L_1$ and all $y\in L_1$ with $x\curlyvee y\equiv\top$ we have
    \begin{align*}
        \mathsf U(\mathsf S(f))(\Phi_{\mathsf L_1}(x))&=\mathsf U(\mathsf S(f))([(x,y)])=[(f(x),f(y))]=\Phi_{\mathsf L_2}(f(x)).
    \end{align*}
    Then $\Phi:\mathsf L\rightarrow\Phi_{\mathsf L}$ is a natural isomorphism that witnesses the isomorphism of functors $$\Phi:1_{\textbf{IHL}}\cong \mathsf U\circ \mathsf S.$$
\end{proof}

\begin{theorem}\label{equivalence-02}
    For all $\mathsf A\in \textbf{EHC}_\omega\textbf{A}$ there is an isomorphism $\Psi_{\mathsf A}:\mathsf A\rightarrow \mathsf S(\mathsf U(\mathsf A))$. Moreover, this provides a natural isomorphism of functors $\Psi:1_{\textbf{EHC}_\omega\textbf{A}}\Rightarrow \mathsf S\circ \mathsf U$ given by $\mathsf A\mapsto\Psi_{\mathsf A}$. 
\end{theorem}
\begin{proof}
    Given $x \in A$ , if $z\in\div x$ and $z' \in \div x$ then, by definition, $z \sim z'$  and so $[z]=[z']$. Moreover, $x \curlyvee z \equiv \top$ and so $[x] \curlyvee_{\mathsf U(\mathsf A)} [z] \equiv \top_{\mathsf U(\mathsf A)}$, showing that $([x],[z]) \in S^{C_\omega}_{\mathsf U(\mathsf A)}$, the domain of $\mathsf S(\mathsf U(\mathsf A))$.
    
    With these considerations, we have a well-defined function $\Psi_{\mathsf A}:A\rightarrow S^{C_\omega}_{\mathsf U(\mathsf A)}$ given, for $x\in A$, by the rule $\Psi_{\mathsf A}(x):=([x],[z])$, where $z\in\div x$ is arbitrary.
    
    To prove that $\Psi_{\mathsf A}$ is an \textbf{EHC$_\omega$A}-morphism, let $x,x'\in A$ and $d\in x\curlywedge x'$. Also let $d'\in A$ such that $d\curlyvee d'\equiv\top$. Observe that for all $y,y''\in A$ with $x\curlyvee y\equiv\top$ and $x'\curlyvee y'\equiv\top$, we have $([d],[d'])\in([x],[y])\curlywedge([x'],[y''])$ in $\mathsf S(\mathsf U(\mathsf A))$, which means that $\Psi_{\mathsf A}(d)\in \Psi_{\mathsf A}(x)\curlywedge \Psi_{\mathsf A}(x')$. Similarly we prove that $d\in x\curlyvee x'$ implies that $\Psi_{\mathsf A}(d)\in \Psi_{\mathsf A}(x)\curlyvee \Psi_{\mathsf A}(x')$ and $d\in x\multimap x'$ implies $\Psi_{\mathsf A}(d)\in \Psi_{\mathsf A}(x)\multimap \Psi_{\mathsf A}(x')$, proving that $\Psi_{\mathsf A}$ is an \textbf{IHL}-morphism. The next step is to prove that $\Psi_{\mathsf A}(\div x)\subseteq\div\Psi_{\mathsf A}(x)$. Observe first that 

    $\begin{array}{lll}
         \div\Psi_{\mathsf A}(x)&=&\div ([x],[z]) = \big\{([z],[y])\in S^{C_\omega}_{\mathsf U(\mathsf A)} \ : \  [y]\preceq_{\mathsf U(\mathsf A)}[x]\big\}\\[1mm]
         &=& \big\{([z],[y])\in S^{C_\omega}_{\mathsf U(\mathsf A)} \ : \  y\preceq x\big\} \mbox{, for any $z\in\div x$}.    
    \end{array}$
    
    Now, let $z\in\div x$. Since $\div z\subseteq\div\div x$ and $\div\div x\preceq x$, we get $\div z\preceq x$. Then for every $y\in\div z$ it holds that $y \preceq x$ and so
    $$\Psi_{\mathsf A}(z)=([z],[y])\in\div\Psi_{\mathsf A}(x),$$
proving that $\Psi_{\mathsf A}:\mathsf A\rightarrow \mathsf S(\mathsf U(\mathsf A))$ is an \textbf{EHC$_\omega$A}-morphism.

    If $\Psi_{\mathsf A}(x)=\Psi_{\mathsf A}(y)$ then $([x],[x'])=([y],[y'])$ for all $x'\in\div x$ and $y'\in\div y$, which means $x\sim y$ and $\div x\sim \div y$. By Axiom E4 we have $x=y$ and so $\Psi_{\mathsf A}$ is injective. Also, if $([x],[y])\in S^{C_\omega}_{\mathsf U(\mathsf A)}$ then $[x]\curlyvee_{\mathsf U(\mathsf A)} [y] \equiv \top_{\mathsf U(\mathsf A)}$ which implies that $x\curlyvee y\equiv\top$. By Axiom E3 there exist $w\in A$ such that $x\sim w$ and $y\sim \div w$. Therefore, given $w'\in\div w$ we get $\Psi_{\mathsf A}(w)=([w],[w'])=([x],[y])$ and so $\Psi_{\mathsf A}$ is surjective.
    
    Therefore $\Psi_{\mathsf A}:\mathsf A\rightarrow \mathsf S(\mathsf U(\mathsf A))$ is an isomorphism which is natural in the sense that for an \textbf{EHC$_\omega$A}-morphism $g:\mathsf A_1\rightarrow \mathsf A_2$, the following diagram commutes: 
    $$\xymatrix@!=4.5pc{
    \mathsf A_1\ar[r]^{\Psi_{\mathsf A_1}}\ar[d]_{g} & \mathsf S(\mathsf U(\mathsf A_1))\ar[d]^{\mathsf S(\mathsf U(g))} \\
    \mathsf A_2\ar[r]_{\Psi_{\mathsf A_2}} & \mathsf S(\mathsf U(\mathsf A_2))
    }$$
    In fact, for all $x\in A_1$ and $y\in\div x$ we have
    \begin{align*}
        \mathsf S(\mathsf U(g))(\Psi_{\mathsf A_1}(x))&=\mathsf S(\mathsf U(g))([x],[y])=(\mathsf U(g)([x]),\mathsf U(g)([y]))=([g(x)],[g(y)])=\Psi_{A_2}(g(x)).
    \end{align*}
    Then $\Psi:\mathsf A\rightarrow\Psi_{\mathsf A}$ is a natural isomorphism that witnessess the isomorphism of functors $$\Phi:1_{\textbf{EHC}_\omega\textbf{A}}\cong \mathsf S\circ \mathsf U.$$
\end{proof}

Combining Theorems~\ref{equivalence-01} and~\ref{equivalence-02} we arrive at our main result:

\begin{theorem}\label{equivalence-03}
    The functors $\mathsf S:\textbf{IHL}\rightarrow \textbf{EHC}_\omega\textbf{A}$ and $\mathsf U:\textbf{EHC}_\omega\textbf{A}\rightarrow \textbf{IHL}$ establish an equivalence of categories.
\end{theorem}

The latter result shows that any  enriched hyper $C_\omega$ algebra has a representation as a swap structure over a Sette implicative hyperlattice.

\section{Extending the results to axiomatic extensions of  $C_\omega$} \label{aect:axiom-ext}

In this section, we extend our previous results on $C_\omega$ to certain axiomatic extensions of this logic. We will only consider two simple but interesting cases: the logics $C_{min}$  and $C_\omega^+$. Both logics are defined over the signature $\Sigma_\omega$ of $C_\omega$.

\subsection{The logic $C_{min}$}

The logic $C_{min}$ was introduced by Carnielli and Marcos in~\cite{car:mar:99} as a way to keep  $C_\omega$ closer to the limit of the hierarchy of da Costa systems $C_n$, for $1 \leq n < \omega$. In order to do this, the logic $C_{min}$ is defined as the axiomatic extension of $C_\omega$ by adding the Peirce/Dummett law $(PL): \ \varphi \vee (\varphi \to \psi)$. In this way, the positive basis of  $C_{min}$ coincides with positive classical logic, in contrast to  $C_\omega$ which is based on positive intuitionistic logic.

With this move,  $C_{min}$ is semantically characterized in terms of bivaluations $b:For(\Sigma_\omega) \to \{0,1\}$ satisfying the standard clauses for valuations for  classical logic  (i.e., the standard truth-tables) for the binary connectives, and two clauses for negation, which reflect the validity of axioms (EM) and (cf), namely
$$\mbox{\bf (val 1)} \ \ b(\varphi)=0 \ \mbox{ implies that } b(\neg\varphi)=1; \ \ \mbox{ and} \ \ \mbox{\bf (val 2)} \ \ b(\neg\neg\varphi)=1 \ \mbox{ implies that } b(\varphi)=1$$
for every formula $\varphi$. In this way, $C_{min}$ is both a `syntactic limit' of the other calculi $C_n$ (by retaining all shared axioms, including (PL), a theorem provable in all of them) and a `semantic limit', in the sense that $C_{min}$ is semantically characterized by bivaluations satisfying {\em exactly} the five clauses common to all bivaluations for the other calculi $C_n$ (see~\cite{dac:alv:77,lop:alv:80}). However, while all other calculi $C_n$ are finitely trivializable, $C_{min}$ is not, meaning it does not yet constitute the deductive limit of this family of paraconsistent systems.\footnote{The deductive limit of the hierarchy $C_n$, for $1 \leq n < \omega$, is a non-finitary logic called $C_{Lim}$, see~\cite{car:mar:99} for details.}

It is easy to adapt our semantical framework to  deal with $C_{min}$. The first observation to be made is that, since $C_{min}$ is based on positive classical logic, the underlying algebraic structures are now {\em classical implicative lattices}.\footnote{This name was taken from Curry, see his book~\cite{C77}.} A classical implicative lattice is simply an implicative lattice \La\ such that $a \vee (a \to b)=1$ for every $a,b \in \La$ (or, equivalently, $((a \to b) \to a) \to a=1$ for any $a,b \in \La$). This variety of algebras characterizes precisely positive classical logic. The generalization to hyperalgebras is straightforward:

\begin{definition} [Classical implicative hyperlattices] \label{def:c-ilattices} A {\em classical implicative hyperlattice} (or a {\em CIHL}) is an IHL $\mathsf L=\langle L,\curlywedge,\curlyvee,\multimap  \rangle$ such that, for every $x,y,z,w \in L$:
\begin{description}
    \item[(I4)] $z \in x \multimap y$ and $w \in x \curlyvee z$ implies that $w \in \top$.
\end{description}
\end{definition}

\noindent Clearly, (I4) is equivalent to the following condition:
\begin{description}
    \item[(I4)'] $x \curlyvee (x \multimap y) \equiv \top$,  for every $x,y\in L$.
\end{description}

\begin{definition} [Hyperalgebras  for $C_{min}$] \label{def:HCmin}
A {\em Hyperalgebra for  $C_{min}$} (or {\em hyper  $C_{min}$ algebra}, or simply a HC$_{min}$A) is a hyper $C_\omega$ algebra $\mathsf H=\langle H,\curlywedge,\curlyvee,\multimap, \div  \rangle$  such that the reduct~\mbox{$\langle H,\curlywedge,\curlyvee,\multimap\rangle$}  is a CIHL.
\end{definition}

\noindent
The consequence relation with respect to  HC$_{min}$As, defined as in Definition~\ref{def-sem-SHCw} by restricting to HC$_{min}$As, will be denoted by $\models_{\mathbb{HC}_{min}}$.

As expected, the swap structures for $C_{min}$ are the ones for $C_\omega$ which are induced by classical implicative lattices, while its hyper swap structures are the corresponding ones induced by classical implicative hyperlattices. In formal terms:

\begin{definition} [Swap structures for $C_{min}$]  Let $\mathsf L=\langle L, \wedge,\vee,\to\rangle$ be a classical implicative lattice.  The swap structure for $C_{min}$ over $\mathsf L$ is  $\Sw_0(\mathsf L)$, the swap structure for $C_\omega$ over $\mathsf L$ as introduced in Definition~\ref{def-swap-Cw}.
\end{definition}

\noindent As in the case of $C_\omega$, the Nmatrix associated to  $\Sw_0(\mathsf L)$ is $\mathcal{M}_0(\mathsf L)=\langle \Sw_0(\mathsf L),D_{\mathsf L} \rangle$, and the consequence relation generated by the class of Nmatrices of the form $\mathcal{M}_0(\mathsf L)$, for \La\ a classical implicative lattice, will be denoted by $\models_{C_{min}}^{SW}$. Then, the following result easily follows from Theorem~\ref{Sound-comple0} and the definitions above:

\begin{theorem} [Soundness and completeness of  $C_{min}$ w.r.t. hyperstructures semantics, version~1] \label{Sound-complemin0} \ \\
Let $\Gamma \cup \{\varphi\}$ be a set of formulas over $\Sigma_\omega$. The following assertions are equivalent:
\begin{enumerate}
    \item $\Gamma \vdash_{C_{min}} \varphi$;
    \item $\Gamma \models_{\mathbb{HC}_{min}} \varphi$;
    \item $\Gamma \models_{C_{min}}^{SW} \varphi$.
\end{enumerate}
\end{theorem}

\noindent The adaptation of our results to hyper swap structures for $C_{min}$ (and so, the definition of an equivalence of categories) is also straightforward.

\begin{definition} [Hyper Swap structures for $C_{min}$] Let $\mathsf L=\langle L, \curlywedge,\curlyvee,\multimap\rangle$ be a CIHL.
The hyper  swap structure for $C_{min}$ over $\mathsf L$ is  $\mathsf S(\mathsf L)$, the hyper swap structure for $C_\omega$ over $\mathsf L$ as introduced in Definition~\ref{def-sem-swap-Cw}.
\end{definition}

\noindent By restricting the consequence relation $\models_{C_\omega}^{HSW}$ induced by hyper swap structures for  $C_\omega$ (recall Definition~\ref{def-sem-swap-Cw}) to hyper swap structures for $C_{min}$, we obtain a consequence relation which will be denoted by $\models_{C_{min}}^{HSW}$. Then, from Theorem~\ref{Sound-comple} and the definitions above, we get the following:

\begin{theorem} [Soundness and completeness of  $C_{min}$ w.r.t. hyperstructures semantics, version~2] \label{Sound-comple2} \ \\
Let $\Gamma \cup \{\varphi\}$ be a set of formulas over $\Sigma_\omega$. The following assertions are equivalent:
\begin{enumerate}
    \item $\Gamma \vdash_{C_{min}} \varphi$;
    \item $\Gamma \models_{\mathbb{HC}_{min}} \varphi$;
    \item $\Gamma \models_{C_{min}}^{HSW} \varphi$.
\end{enumerate}
\end{theorem}

\begin{definition}[Enriched Hyper C$_{min}$ Algebras] 
    Let $\mathsf A=\langle A,\curlywedge,\curlyvee,\multimap,\div\rangle$ be a HC$_{min}$A. We say that $\mathsf A$ is an {\em enriched hyper C$_{min}$ algebra (EHC$_{min}$A)} if it an EHC$_\omega$A, i.e., it satisfies axioms E0-E4 from Definition~\ref{def:EHCw}.
\end{definition}


    \noindent The category \textbf{EHC$_{min}$A} is the full subcategory of  \textbf{EHC$_\omega$A} where the objects are EHC$_{min}$As. In turn, \textbf{CIHL} is the full subcategory of \textbf{IHL} whose objects are classical implicative hyperlattices. From Theorem~\ref{equivalence-03} and by the definitions above, the proof of the following result is immediate:

\begin{theorem}
    The functors $\mathsf S:\textbf{IHL}\rightarrow \textbf{EHC}_\omega\textbf{A}$ and $\mathsf U:\textbf{EHC}_\omega\textbf{A}\rightarrow \textbf{IHL}$ can be restricted to functors  $\bar{\mathsf S}:\textbf{CIHL}\rightarrow \textbf{EHC}_{min}\textbf{A}$ and $\bar{\mathsf U}:\textbf{EHC}_{min}\textbf{A}\rightarrow \textbf{CIHL}$, which establish an equivalence of categories.
\end{theorem}

\subsection{The logic $C_\omega^+$}

One of the most interesting features of the logic $C_\omega$ is that, while it is based on positive intuitionistic logic, it is paraconsistent with respect to its primitive negation. The basic properties of such paraconsistent negation $\neg$ are two principles enjoyed by classical negation: excluded middle, (EM), and double negation elimination, (cf). 

Observe that (cf) forces the negation to be  non-deterministic in the swap structures semantics for $C_\omega$. Indeed, given $z=(z_1,z_2)$ in $S_{\mathsf L}$ representing values assigned to $(\varphi,\neg\varphi)$ in $\La\times \La$ then, by applying $\breve{\neg}$ to $z$, it is obtained the set $\{u\in S_{\mathsf L} \ : \ u_1=z_2 \mbox{ and }  u_2 \leq z_1 \}$, according to Definition~\ref{def-swap-Cw}. While the first coordinate of the elements of $\breve{\neg} z$ is the same (namely, $z_2$), which rightly `reads' the current value of $\neg\varphi$ in \La, the second coordinate, which represents the values in \La\ to be assigned to $\neg\neg \varphi$, is ambiguous in the following sense: any value for $\neg\neg\varphi$ less or equal than the original value assigned in \La\ to $\varphi$ (namely, $z_1$) is acceptable. This is a direct consequence of axiom (cf), and fully justifies the condition `$u_2 \leq z_1$' in the definition of $\breve{\neg}z$ (recalling that, in an implicative lattice, $x \to y =1$ iff $x \leq y$). Compare this situation with that in Nelson's logic N4: unlike our case, N4 validates not only (cf) but also its converse axiom $\varphi \to \neg\neg \varphi$. Because of this, the twist structures for N4, which are also induced by implicative lattices, and also represent the values assigned to $(\varphi,\neg\varphi)$ over such lattices, are such that the interpretation of negation is {\em deterministic}, given as follows: $\breve{\neg}(z_1,z_2)=(z_2,z_1)$.\footnote{As a matter of fact, twist structures are algebraic structures and consequently all their operations are deterministic.} The fact that the value of the second coordinate --- the value to be assigned to $\neg\neg\varphi$ --- is $z_1$ reflects the validity of the law $\varphi \leftrightarrow \neg\neg\varphi$.

Motivated by this, we introduce the following  axiomatic extension of  $C_\omega$:

\begin{definition}
 The logic  $C_\omega^+$ is defined by the Hilbert calculus over $\Sigma_\omega$ obtained by adding to   $C_\omega$ axiom (ce): $\varphi \to \neg\neg\varphi$.
\end{definition}

\noindent
The logic $C_\omega^+$ is still paraconsistent, as it will be shown in Corollary~\ref{Cw+-paracons}. Given that $\neg\neg\varphi$ is equivalent to $\varphi$ in $C_\omega^+$, the multioperator for the corresponding swap structures is deterministic, and it is defined as in the twist structures for N4 above mentioned, see Definition~\ref{def:swap:Cw+} below.

Before introducing a class of  hyperalgebraic models for $C_\omega^+$ by adapting the results from $C_\omega$, it will be shown that it is not possible to characterize this logic by means of a single finite Nmatrix; in particular, it cannot be characterized by a single finite matrix. In order to do this, the original proof of G\"odel from 1932 stating that intuitionistic logic cannot be characterized by a single finite matrix (see~\cite{Godel1932}), and generalized to Nmatrices in~\cite{leme:con:lop:24}, will be adapted to  $C_\omega^+$.

For $n \geq 1$, let $G_n$ be the following formula over $\Sigma_\omega$ (taken from~\cite{Godel1932}):
$$G_n := \bigvee_{1 \leq i < j \leq n+1} (p_i \to p_j) \land (p_j \to p_i).$$
By adapting the proof given by G\"odel for intuitionistic logic, it will be proven that no formula $G_n$ can be proven  in $C_\omega^+$.

\begin{definition} \label{defMg}
  Let $\mathcal{M}_G$ be an infinite matrix with domain $\omega= \{0,1, 2, 3, \ldots\}$, where  $\{0,1\}$ is the set of designated values. The operators in $\mathcal{M}_G$ are defined  as follows, for every $x,y \in \omega$:
     $$x \vee y := \min(x,y); \hspace{1.4cm} x \wedge y := \max(x,y)$$
     $$
              x \rightarrow y = \left\{ \begin{array}{ll}
              			                                 0 & \textrm{if } x \geq y \\
                    			                           y & \textrm{otherwise}
																						\end{array} \right.
		  \hspace{1cm}
              \neg x = \left\{ \begin{array}{ll}
              			                                 0 & \mbox{if } x = 0 \\
                                                          2 & \mbox{if } x = 1 \\
                    			                           	1 & \textrm{otherwise}
																				\end{array} \right.
			$$
\end{definition}

\

\begin{proposition}
     $\mathcal{M}_G$ is a paraconsistent model of $C_\omega^+$ which invalidates every formula $G_n$, for $n \geq 1$.
\end{proposition}
\begin{proof}  It is immediate to see that $\mathcal{M}_G$ is a model of $C_\omega^+$. It is paraconsistent, since $0 \in D$ and $\neg 0=0 \in D$. Now, consider the following valuation over $\mathcal{M}_G$:  $v(p_i) = i$, for every $1 \leq i < \omega$. By induction on $n\geq 1$, it is easy to prove that $v(G_n)=2$, for any  $n \geq 1$. Hence, $\mathcal{M}_G$ does not validate any formula $G_n$.     
\end{proof}

%

\begin{corollary}  \label{lem-god-0}
None of the formulas $G_n$ is provable in $C_\omega^+$.
\end{corollary}

\begin{corollary}  \label{Cw+-paracons}
$C_\omega^+$ is a paraconsistent logic w.r.t. its negation $\neg$.
\end{corollary}

\begin{proposition} \label{lem-god-1}
Suppose that $\mathcal{M} = \langle \mathsf{A},D \rangle$ is a finite Nmatrix such that the domain $A$ of $\mathsf{A}$ has exactly $n\geq 1$ elements, and $\mathcal{M}$ is a model for $C_\omega^+$. Then, the formula $G_n$ is valid in  $\mathcal{M}$.
\end{proposition}
\begin{proof}
    It follows by adapting the proof for intuitionistic logic found in~\cite{leme:con:lop:24}. For any connective $\#$ of $\Sigma_\omega$ let $\tilde \#$ be its interpretation in $\mathcal{M}$. Since  $\mathcal{M}$ models positive intuitionistic logic and  $\varphi \to \varphi$ is a theorem of this logic then, for every $x,y \in A$ it holds: (i)~$x \,\tilde{\vee}\, y \subseteq D$ if either $x \in D$ or $y \in D$; (ii)~$x \,\tilde{\wedge}\, y \subseteq D$ if both $x \in D$ and $y \in D$; and (iii)~$x \,\tilde{\to}\, x \subseteq D$. 
Let $v$ be a valuation over $\mathcal{M}$. Given that $A$ has $n\geq 1$ elements,  by the  Pigeonhole Principle it follows that $v(p_i)=v(p_j)$ for some $1 \leq i < j \leq n+1$. By the observations(i)-(iii),   $v((p_i \to p_j) \land (p_j \to p_i)) \in  D$ and so $v(G_n) \in \mathcal{D}$.
\end{proof}

\begin{theorem} \label{unchar-Cw+}
The logic $C_\omega^+$ cannot be characterized by a single finite-valued Nmatrix.
\end{theorem}
\begin{proof}
Suppose by contradiction that there exists  a finite Nmatrix $\mathcal{M}$ with exactly $n\geq 1$ elements such that $\mathcal{M}$  characterizes the logic $C_\omega^+$. In particular, any formula valid in $\mathcal{M}$ can be proven in $C_\omega^+$. By Proposition~\ref{lem-god-1}, the formula $G_n$ is valid in  $\mathcal{M}$, and so $G_n$ is provable in $C_\omega^+$. But this contradicts Corollary~\ref{lem-god-0}. Therefore, $C_\omega^+$ cannot be characterized by a single finite-valued Nmatrix.
\end{proof}

\noindent In order to obtain an adequate semantics for  $C_\omega^+$, let us now adapt our hyperalgebraic framework to this  system.

\begin{definition} [Hyperalgebras  for $C_\omega^+$]
A {\em Hyperalgebra for  $C_\omega^+$} (or {\em hyper  $C_\omega^+$ algebra}, or simply a HC$_\omega^+$A) is a hyper $C_\omega$ algebra $\mathsf H=\langle H,\curlywedge,\curlyvee,\multimap, \div  \rangle$  such that condition (H2) is replaced by the following:
\begin{description}
    \item[(H3)] $y \in \div x$ and $w\in \div y$ implies that  $w \equiv x$
\end{description}
    or, equivalently, by the condition
\begin{description}
    \item[(H3')] $\div \div x \equiv x$.
\end{description}
\end{definition}

\noindent
The consequence relation with respect to  HC$_\omega^+$As is given by restricting Definition~\ref{def-sem-SHCw}  to  HC$_\omega^+$As, and it will be denoted by $\models_{\mathbb{HC}_\omega^+}$.

The swap structures for $C_\omega^+$ require slight adjustments with respect to the ones for $C_\omega$:

\begin{definition} [Swap structures for $C_\omega^+$] \label{def:swap:Cw+}  Let $\mathsf L=\langle L, \wedge,\vee,\to\rangle$ be an implicative lattice.  The swap structure for $C_\omega^+$ over $\mathsf L$, denoted by  $\Sw_0^+(\mathsf L)$, is defined as the swap structure  $\Sw_0(\mathsf L)$ for $C_\omega$ over $\mathsf L$ introduced in Definition~\ref{def-swap-Cw}, by replacing the multioperator $\breve{\neg}$ by the following:
\begin{align*}
    \bar{\neg} z&:=\{u\in S_{\mathsf L} \ : \ u_1=z_2 \mbox{ and }  u_2 = z_1  \}= \{(z_2,z_1)\}.
\end{align*}
\end{definition}

\noindent This means that the hyperoperator for negation in the swap structures for $C_\omega^+$ is deterministic, and it is defined as in the case of twist structures for Nelson's N4.

Let  $\mathcal{M}_0^+(\mathsf L)=\langle \Sw_0^+(\mathsf L),D_{\mathsf L} \rangle$ be  the Nmatrix associated to  $\Sw_0^+(\mathsf L)$ as in the case of $C_\omega$, and let $\models_{C_\omega^+}^{SW}$ be the consequence relation generated by this class of Nmatrices. Then, the following result easily follows from Theorem~\ref{Sound-comple0} and the definitions above:

\begin{theorem} [Soundness and completeness of  $C_\omega^+$ w.r.t. hyperstructures semantics, version~1] \label{Sound-compleCw+0} \ \\
Let $\Gamma \cup \{\varphi\}$ be a set of formulas over $\Sigma_\omega$. The following assertions are equivalent:
\begin{enumerate}
    \item $\Gamma \vdash_{C_\omega^+} \varphi$;
    \item $\Gamma \models_{\mathbb{HC}_\omega^+} \varphi$;
    \item $\Gamma \models_{C_\omega^+}^{SW} \varphi$.
\end{enumerate}
\end{theorem}

\noindent We will proceed now to adapt our previous results to hyper swap structures for $C_\omega^+$, including  the definition of an equivalence of suitable categories.

\begin{definition} [Hyper Swap structures for $C_\omega^+$] Let $\mathsf L=\langle L, \curlywedge,\curlyvee,\multimap\rangle$ be an IHL.
The hyper  swap structure for $C_\omega^+$ over $\mathsf L$, denoted by   $\mathsf S^+(\mathsf L)$, is defined as the hyper swap structure  $\mathsf S(\mathsf L)$ for $C_\omega$ over $\mathsf L$ introduced in Definition~\ref{def-sem-swap-Cw}, by replacing the multioperator $\div$ by the following:
\begin{align*}
    \bar{\div} z&:=\{u\in S^{C_\omega}_{\mathsf L} \ : \ u_1=z_2 \mbox{ and }  u_2 \equiv z_1 \}.
\end{align*}
\end{definition}

\noindent Observe that, for $z\in S^{C_\omega}_{\mathsf L}$,
    \[
        \bar{\div}\bar{\div} z=\{u\in S^{C_\omega}_{\mathsf L} \ : \ u_1\equiv z_1\mbox{ and }u_2\equiv z_2\}.
    \]

The consequence relation $\models_{C_\omega^+}^{HSW}$ induced by hyper swap structures for  $C_\omega^+$ is defined analogously to the case of $C_\omega$ (recall Definition~\ref{def-sem-swap-Cw}). The proof of the following result is straightforward by adapting the proof of Theorem~\ref{Sound-comple} according to the definitions above:

\begin{theorem} [Soundness and completeness of  $C_\omega^+$ w.r.t. hyperstructures semantics, version~2] \label{Sound-comple3} \ \\
Let $\Gamma \cup \{\varphi\}$ be a set of formulas over $\Sigma_\omega$. The following assertions are equivalent:
\begin{enumerate}
    \item $\Gamma \vdash_{C_\omega^+} \varphi$;
    \item $\Gamma \models_{\mathbb{HC}_\omega^+} \varphi$;
    \item $\Gamma \models_{C_\omega^+}^{HSW} \varphi$.
\end{enumerate}
\end{theorem}

\noindent Finally, an equivalence of categories can be obtained, by using enriched hyper algebras once again.

\begin{definition}[Enriched Hyper C$_\omega^+$ Algebras] 
    Let $\mathsf A=\langle A,\curlywedge,\curlyvee,\multimap,\div\rangle$ be a HC$_\omega^+$A. We say that $\mathsf A$ is an {\em enriched hyper C$_\omega^+$ algebra (EHC$_\omega^+$A)} if it an EHC$_\omega$A, i.e., it satisfies axioms E0-E4 from Definition~\ref{def:EHCw}.
\end{definition}

\noindent Let $\mathsf A$ be an EHC$_\omega^+$A. By using the notation from  Definition~\ref{def:EHCw} let $[x]:=\{y \ : \ x\sim y\}$, for $x\in A$, and let $$\mathsf U^+(\mathsf A):=A/{\sim}=\{[x] \ : \ x\in A\}.$$
Let \textbf{EHC$_\omega^+$A} be the full subcategory of  \textbf{EHC$_\omega$A} where the objects are EHC$_\omega^+$As. Observe that the hyper swap construction for $C_\omega^+$ induces a functor $\mathsf S^+:\textbf{IHL}\rightarrow \textbf{EHC}_\omega^+\textbf{A}$. In turn, it is easy to adapt Proposition~\ref{functor U- Cw} to $C_\omega^+$ obtaining, by using the new definitions, a functor $\mathsf U^+:\textbf{EHC}_\omega^+\textbf{A}\rightarrow \textbf{IHL}$.  By adapting the proof of Theorem~\ref{equivalence-03} and by the definitions introduced above,  the following result is obtained:

\begin{theorem}
    The functors $\mathsf S^+:\textbf{IHL}\rightarrow \textbf{EHC}_\omega^+\textbf{A}$ and $\mathsf U^+:\textbf{EHC}_\omega^+\textbf{A}\rightarrow \textbf{IHL}$  establish an equivalence of categories.
\end{theorem}

\section{Final Remarks} \label{sect:final-remarks}

In this paper we introduce the novel notion of hyper swap structures, which are basically swap structures (i.e., hyperalgebras of a special kind), but defined over hyperalgebras instead of ordinary algebras. The aim of this generalization is to transfer, to the hyperalgebraic context, the Kalman functor associated to twist structures in the framework of algebraic logic. As a case example, we define hyper swap structures for da Costa paraconsistent logic $C_\omega$, which are defined over Sette implicative hyperlattices. As expected, the logic $C_\omega$ is semantically characterized by this class of hyperalgebras (see Theorem~\ref{Sound-comple}),  just as it is by the class of swap structures induced by implicative lattices, as recently proven in~\cite[Theorem~1]{con:gol:rob:2025}.  

A key feature of our construction is that the class of hyper swap structures for $C_\omega$  can be abstracted via the notion of enriched hyperalgebras for $C_\omega$. Specifically, every enriched hyperalgebra admits a representation as a hyper swap structure by means of the Kalmar functor $\mathsf S$ and its inverse functor $\mathsf U$ introduced here. We obtained analogous results for two interesting axiomatic extensions of $C_\omega$: $C_{min}$, introduced in~\cite{car:mar:99}, and  the logic $C_\omega^+$.
This mirrors the standard twist constructions, such as the Kalmar functor linking implicative lattices and twist structures for Nelson’s logic N4 (see~\cite[Chapter~8]{odin:08}), where twist structures for N4 are abstracted to the variety of N4-lattices. An important direction for future research would be to establish a more intrinsic axiomatization of the class of enriched hyperalgebras for $C_\omega$. 

Extending the framework of hyper swap structures to other non-classical logics --- along with defining an appropriate class of enriched hyperalgebras --- constitutes a natural direction for future research. In this context, we are currently applying our techniques to various paraconsistent logics within the family of LFIs.

\

\

\noindent {\bf Acknowledgments:}
Coniglio acknowledges support by an individual research grant from the National Council for Scientific and Technological Development (CNPq, Brazil), grant 309830/2023-0. All the authors were supported by the S\~ao Paulo Research Foundation (FAPESP, Brazil), thematic project {\em Rationality, logic and probability -- RatioLog}, grant  2020/16353-3. Roberto was supported by a post-doctoral grant from FAPESP, grant 2024/18577-7.

\bibliographystyle{plain}

\end{document}